\documentclass[11pt,reqno]{amsart}
\usepackage{amsmath,amsfonts,amssymb,amsthm,stmaryrd,enumerate,url}
\usepackage{microtype}
\usepackage{bbm,mathrsfs,comment}
\usepackage[utf8]{inputenc}  
\usepackage[english]{babel}  
\usepackage[numbers,sort]{natbib}
\usepackage[usenames,dvipsnames,svgnames,table]{xcolor}
\usepackage[colorlinks=true, pdfstartview=FitV, pagebackref=false]{hyperref}
\usepackage{dsfont}

\usepackage{marginnote}
\usepackage{graphicx}
\usepackage{authblk}
\usepackage[a4paper,vmargin=3cm,inner=2.3cm,outer=2.8cm,  marginparsep=2mm,
    marginparwidth=19mm]{geometry}
 \usepackage[foot]{amsaddr}
\usepackage[normalem]{ulem}

\newtheorem{definition}{Definition}
\newtheorem{theorem}{Theorem}
\newtheorem{lemma}[theorem]{Lemma}
\newtheorem{proposition}[theorem]{Proposition}

\newtheorem{remark}{Remark}

\allowdisplaybreaks

\newcommand{\cC}{\mathcal{C}}

\newcommand{\cF}{\mathcal{F}}

\newcommand{\cL}{\mathcal{L}}

\newcommand{\cN}{\mathcal{N}}

\newcommand{\cP}{\mathcal{P}}

\newcommand{\cT}{\mathcal{T}}

\newcommand{\cX}{\mathcal{X}}

\newcommand{\bR}{\mathbf{R}}

\newcommand{\R}{\mathbb{R}}

\renewcommand{\P}{\mathbb{P}}
\newcommand{\E}{\mathbb{E}}

\def\1{\mathbf 1}
\newcommand{\ind}{{\mathbbm{1}}}

\renewcommand{\hat}{\widehat}

\renewcommand{\bar}{\overline}

\DeclareMathOperator{\Var}{Var}

\newcommand{\tv}{{\textsc{tv}}}

\begin{document}

\title[ ]{Inference in balanced community modulated recursive trees}

\author[]{Anna Ben-Hamou$^\star$}
\address[$\star$]{LPSM, Sorbonne University, Paris}
\email[$\star$]{anna.ben\_hamou@sorbonne-universite.fr}

\author[]{Vasiliki Velona$^\ddag$}
\address[$\ddag$]{Einstein Institute of Mathematics, Hebrew University of Jerusalem}
\email[$\ddag$]{vasiliki.velona@mail.huji.ac.il}

\maketitle

\begin{abstract}
We introduce a random recursive tree model with two communities, called \textit{balanced community modulated random recursive tree}, or BCMRT in short. In this setting, pairs of nodes of different type appear sequentially. Each node of the pair decides independently to attach to their own type with probability $1-q$, or to the other type with probability $q$, and then chooses its parent uniformly within the set of existing nodes with the selected type.  We find that the limiting degree distributions coincide for different $q$. Therefore, as far as inference is concerned, other statistics have to be studied. We first consider the setting where the time-labels of the nodes, i.e., their time of arrival, are observed but their type is not. In this setting, we design a consistent estimator for $q$ and provide bounds for the feasibility of testing between two different values of $q$. Moreover, we show that if $q$ is small enough, then it is possible to cluster the nodes in a way correlated with the true partition, even though the algorithm is exponential in time (in passing, we show that our clustering procedure is intimately connected to the NP-hard problem of minimum fair bisection). In the unlabelled setting, i.e., when only the tree structure is observed, we show that it is possible to test between different values of $q$ in a strictly better way than by random guessing. This follows from a delicate analysis of the sum-of-distances statistic.
\end{abstract}

\vspace*{.3cm}

\section{Introduction}
\subsection{Setting}
A rooted labelled tree is called \textit{recursive} if any path that starts from the root has increasing labels. The most natural model for random recursive trees is the widely studied \emph{uniform recursive tree} (URT) which is generated in this way: initially, the tree consists of one single node with label $1$, and for $i\in\{2,\dots,n\}$, node $i$ attaches to a uniformly chosen node among $\{1,\dots, i-1\}$. We are interested in a community modulated version of this model, where nodes are either of type $A$ or $B$ and attach preferentially to nodes of the same type. Such a model was recently introduced by S.~Bhamidi, R.~Fan, N.~Fraiman, and A.~Nobel in~\cite{bhamidi2022community}, who defined the \emph{community modulated recursive tree} (CMRT) as follows. Initially, the tree consists of one single edge connecting a node of  type $A$ and a node of type $B$. Then, at each subsequent time step, a new node is added to the tree and choses its type with probability $p$. It then attaches to an existing node of the same type with probability $1-q$; otherwise it attaches to a node with different type.

Using an embedding into a continuous time process and techniques from stochastic analysis, the authors in~\cite{bhamidi2022community} determined the limiting degree distribution of CMRT. In particular, when $q\neq 0$ and $p\neq 1/2$, one may use the number of nodes of degree one and two to design consistent estimators for $p$ and $q$. However, when $p=1/2$, i.e., when the two types are approximately equally represented, the preferential parameter $q$ vanishes from the limiting degree distribution and it is not clear how to estimate it. Inspired by this, we introduce a different model, called \emph{balanced community modulated random recursive tree} (BCMRT), which is accessible to discrete-time calculation but also interesting on its own right, as will be described later.

\begin{definition}[BCMRT$(q)$]In the BCMRT model, initially a single edge connects two nodes with time-label $1$, one of type $A$, denoted $(1,A)$, and one of type $B$, denoted $(1,B)$. Then, given $q\in [0,1]$, at each step $n\geq 2$, two new nodes with time-label $n$, one of type $A$ denoted $(n,A)$ and one of type $B$ denoted $(n,B)$, are added to the tree of size $2(n-1)$ in the following way:
\begin{enumerate}
\item the two nodes independently determine if they will connect to a node of their own type with probability $1-q$, or to a node of the other type with probability $q$;
\item they independently choose an existing node of the type selected in step $(1)$ and connect to it, forming a BCMRT of size $2n$.
\end{enumerate}
\end{definition}
Note that in this model, two nodes with the same time-label cannot be connected with edge. In this paper, we mainly focus on the assortative case, i.e., $q\in [0,1/2]$. When $q=1/2$ the community structure disappears: at each step, the two nodes independently attach to a uniformly chosen node in the existing tree, independently of the types. We find that the parameter $q$ vanishes asymptotically in the degree distribution of the BCMRT, as it happens in the CMRT with $p=1/2$.

\begin{figure}
\includegraphics[width=.18\textwidth]{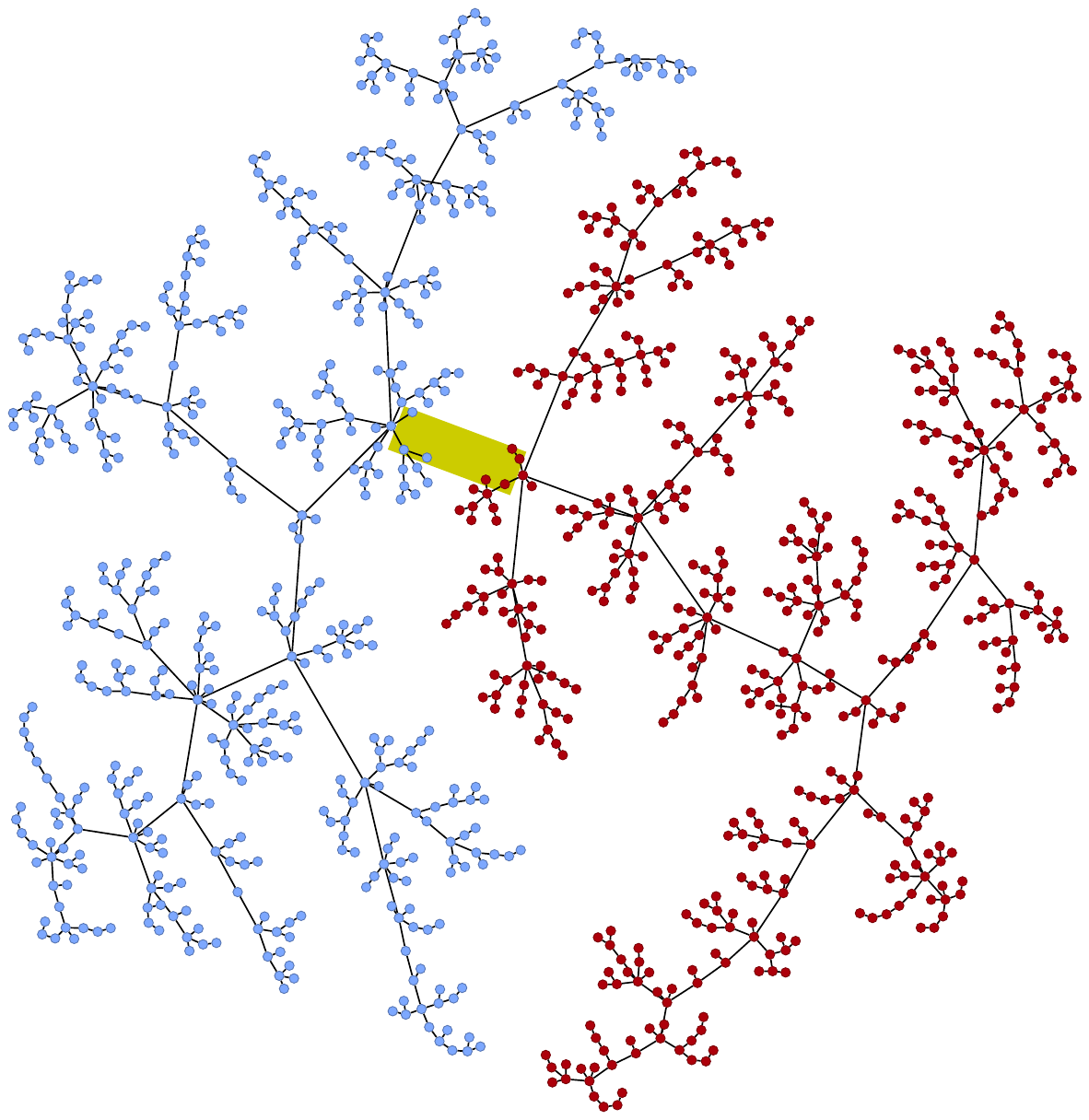}
\includegraphics[width=.18\textwidth]{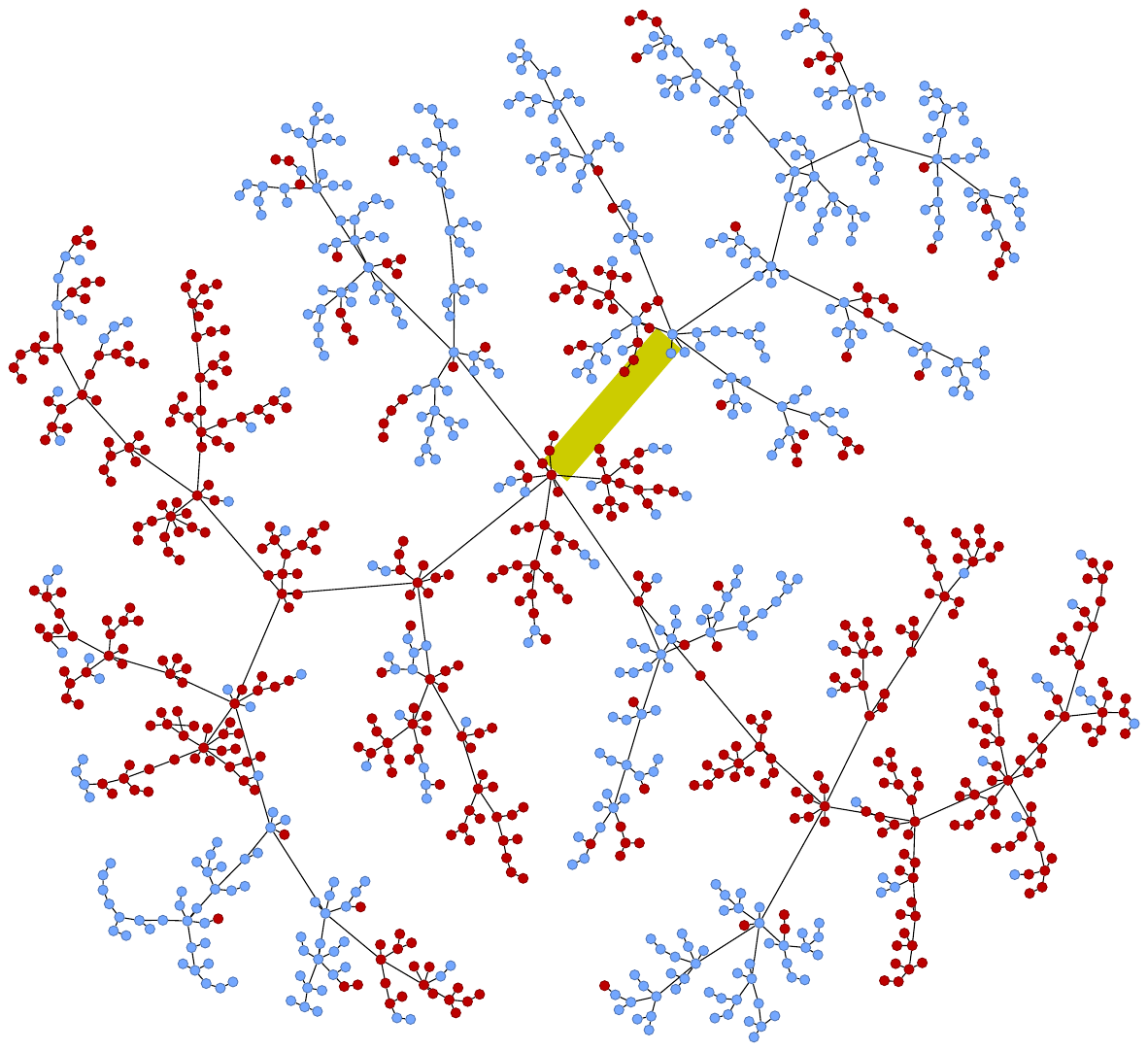}
\includegraphics[width=.18\textwidth]{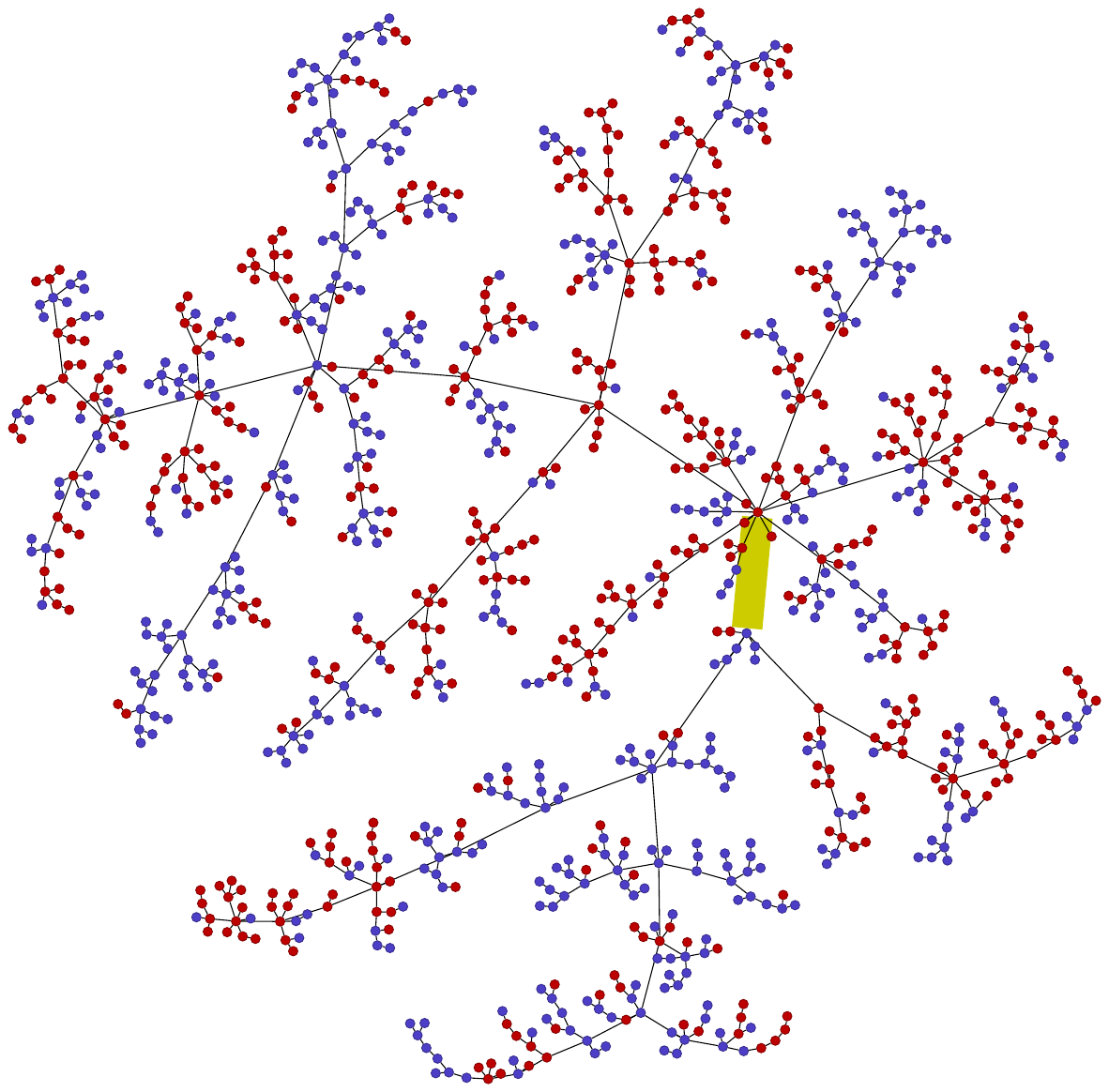}
\includegraphics[width=.18\textwidth]{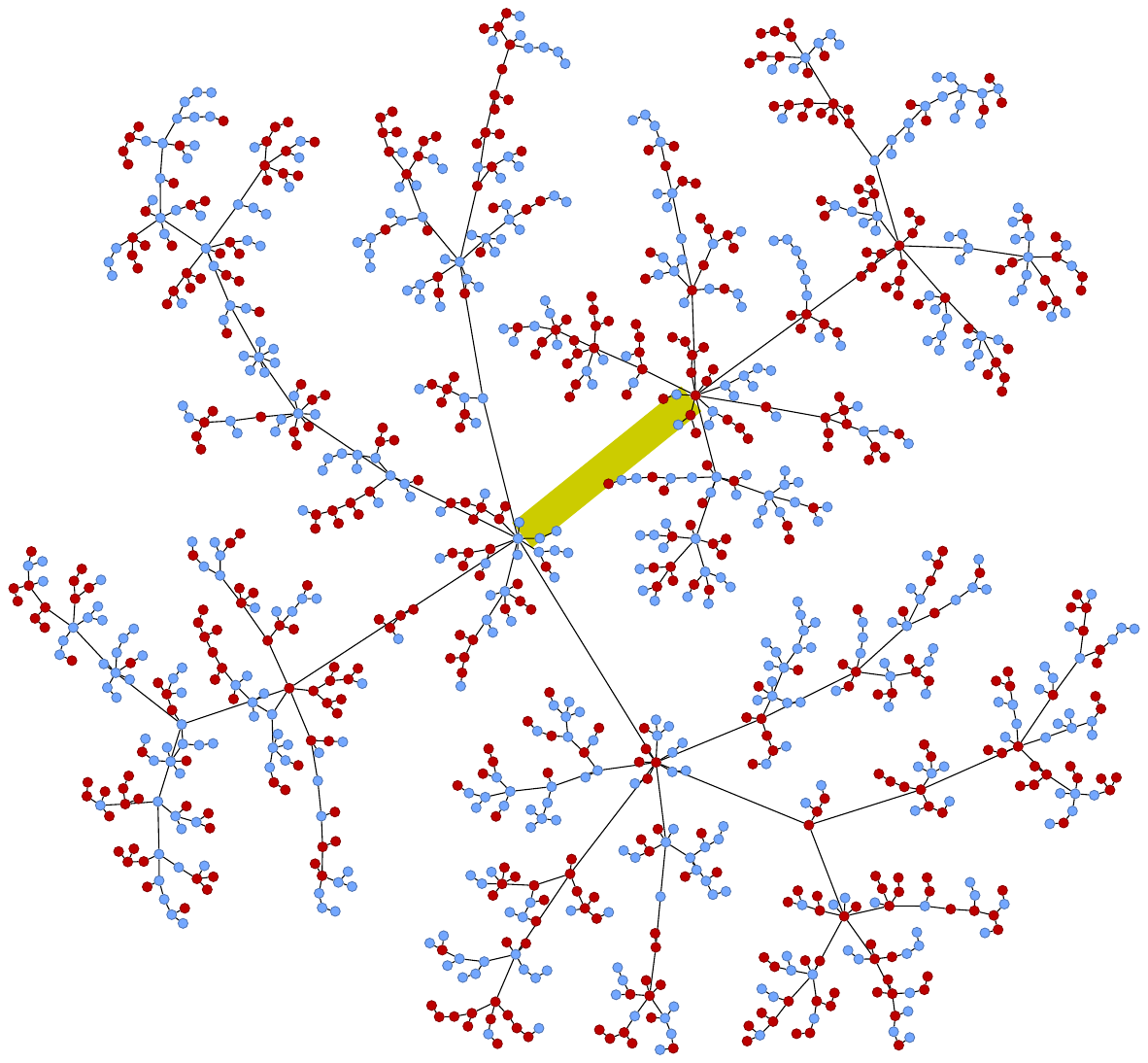}
\includegraphics[width=.18\textwidth]{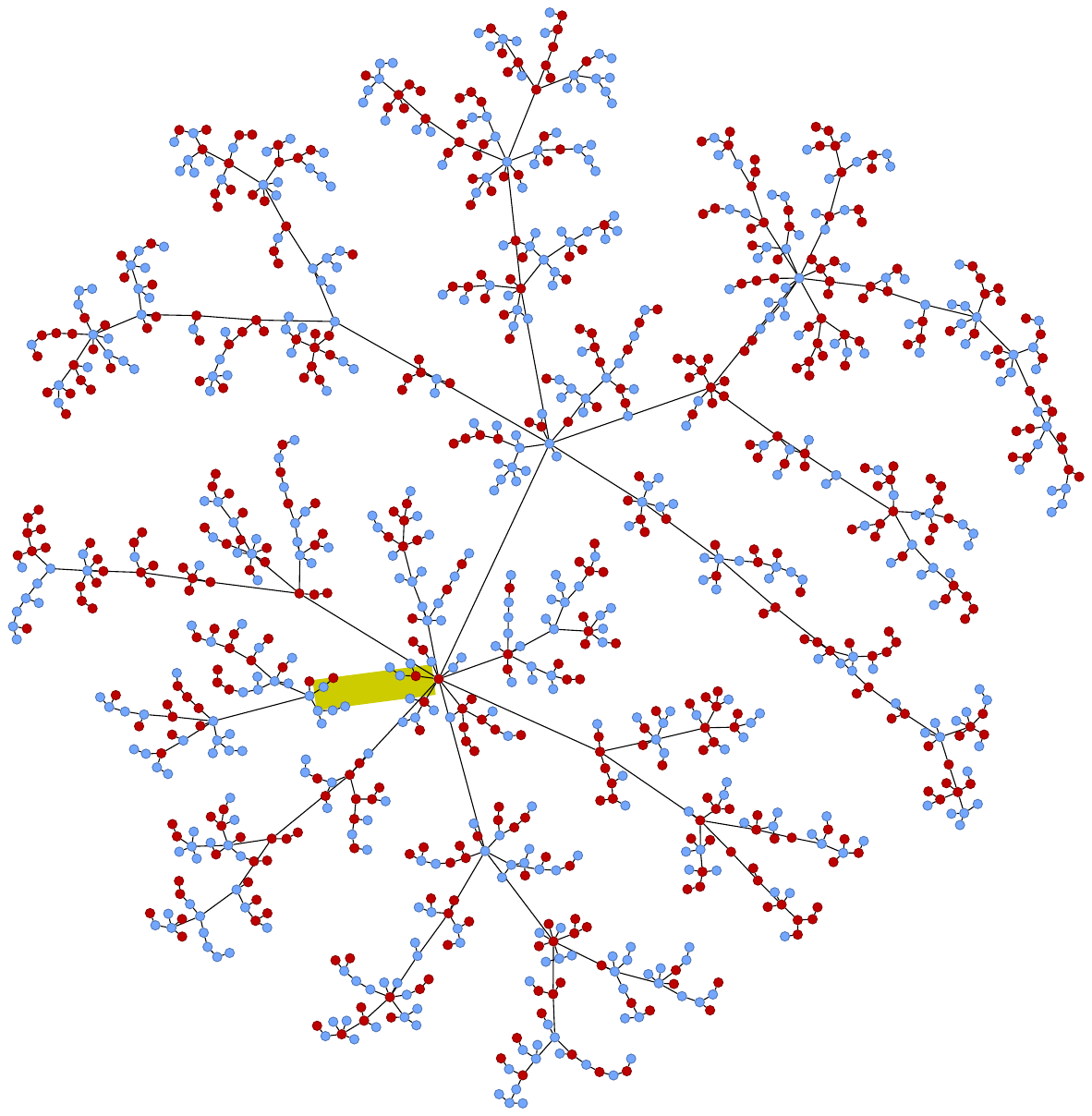} 
  \caption{A simulation of BCMRT(q) with 1000 nodes, for $q\in\left\{0,\frac{1}{8},\frac{1}{4},\frac{3}{8},\frac{1}{2}\right\}$, in order left to right. The colors red and blue represent the two communities, while the yellow edge is the initial root edge.}\label{fig:foobar}
\end{figure}

While the BCMRT was initially inspired by the CMRT, it seems to be of independent interest. It is simple and directly generalises the well-studied URT  with an embedded community structure. Moreover, while being tractable, it exhibits non-trivial behaviour in multiple questions of interest to us. Firstly, can the preferential parameter $q$ be estimated and with what accuracy? Can two different values of $q$ be tested with vanishing error? More ambitiously, is it possible to reconstruct communities better than by random guessing? Of course, the answers to those questions depend on the observational setting. Here we consider three different settings:
\begin{itemize}
    \item \emph{Time-labelled BCMRT}: the tree is observed together with time labels, i.e., nodes are labelled with the integer corresponding to their time of arrival. Each integer $k\in\{1,\dots n\}$ thus appears exactly twice, and the path from each node to its closest root has decreasing labels;
    \item \emph{Rooted unlabelled BCMRT}: the tree is observed without time labels, except for the two first nodes. In other words, the tree comes with a distinguished edge, called the root edge;
    \item \emph{Unrooted unlabelled BCMRT}: only the tree structure is observed, without any additional information.
\end{itemize}

\subsection{Main results}

In Section~\ref{sec:degree}, we put BCMRT into context with respect to other random tree models and compute the limiting degree distribution. This turns out to be the same as in the URT. For $k\geq 1$, let $N_k(n)$ be the number of nodes of degree $k$ at time $n$. Then,
    \[\frac{N_k(n)}{2n}\underset{n\to +\infty}{\overset{\P}{\longrightarrow}}\frac{1}{2^k}\, .\]
    Moreover, $\Var \left(N_1(n)\right)=\Omega \left(n\right)$.
This can be perceived as a negative result, in the sense that we do not expect to infer $q$ using the empirical degree distribution. Further evidence for that is given by the fact that the difference between the expectations of the number of degree-$k$ nodes under different values of $q$ is of constant order while we suspect the variance to be linear in $n$. The latter is only verified for the number of leaves (degree-$1$ nodes).

Section~\ref{sec:bcmrt} is devoted to the time-labelled BCMRT with $q\in [0,1/2]$. Our results in terms of inference and testing can be summarised as follows\footnote{Throughout the paper, the notation $a_n\gg b_n$ stands for $a_n/b_n \underset{n\to +\infty}{\to}+\infty$. Similarly, $a_n\ll b_n$ stands for $a_n/b_n\underset{n\to +\infty}{\to}0$. Moreover $a_n\sim b_n$ stands for $a_n/b_n\underset{n\to +\infty}{\to}1$.}.
    \begin{enumerate}
        \item[(i)] There exists an estimator $\hat{q}_n$ such that $\hat{q}_n{\overset{\P}{\longrightarrow}}q$. The speed of convergence is $(\log n)^{1/2}$ when $0< q\leq \frac{1}{2}$ and $(\log n)^{1/4}$ when $q= \frac{1}{2}$.  
        \item[(ii)] Allowing $q_0$ and $q_1$ to depend on $n$, if
        \[\left(q_1-q_0\right)^2\left(1-q_0-q_1\right)^2\gg \frac{q_1}{\log n}\, ,\]
        then one can test with vanishing error between $q_0$ and $q_1$.
        \item[(iii)] In the case $q_1=1/2$ and $q_0\in [0,1/2[$, the latter condition simplifies to $1/2-q_0\gg (\log n)^{-1/4}$. On the other hand, we show that if  $1/2-q_0\ll n^{-1/2}$, then no test can distinguish $q_0$ and $1/2$ better than random guessing. 
    \end{enumerate}

  We then consider the problem of \textit{clustering}. Assuming that the types of the nodes are hidden, the goal here is to guess the type of each node in such a way that more than half of the nodes are correctly guessed, up to  permutation.
In the time-labelled setting, we show that clustering is possible for small enough $q$. Let us elaborate on the latter. We say that $\sigma$ is a valid coloring of a time-labelled BCMRT of size $2n$ if $\sigma$ assigns the color $A$ or $B$ to each node in such a way that for each pair of nodes with the same time label, one node has color $A$ and the other has color $B$. It is not hard to see that the number of monochromatic edges for the true coloring is concentrated around $2(1-q)n$, with Gaussian fluctuations of order $\sqrt{n}$. We show that, for $q$ small enough (namely $q\leq 1/50$), any valid coloring which is as ``good'' as the true one in terms of monochromatic edges, namely that has more than $2(1-q)n -n^{2/3}$ monochromatic edges, must be correlated with it. This gives a simple (though exponential-time) algorithm to produce a partition of nodes into two groups that performs strictly better than chance: simply go through all valid colorings and output the first one that has more than $2(1-q)n -n^{2/3}$ monochromatic edges. With high probability, the correlation between this coloring and the true one is bounded away from zero. In the Appendix, we provide a proof that shows that our clustering procedure is intimately related to the NP-hard problem of \textit{minimum fair bisection}. In particular, we find that the problem of minimum fair bisection remains NP-hard in recursive trees, which in our case means that finding a coloring that maximises the number of monochromatic edges is NP-hard.

In Sections~\ref{sec:rooted} and~\ref{sec:unrooted} we focus on the rooted unlabelled and the unrooted unlabelled settings. There, we show that testing is possible with non trivial probability of error, for fixed $q$. The latter depends on the total variation distance between BCMRT trees with different parameter $q$, denoted by $d_{\tv}(q_0,q_1)$.
In the rooted unlabelled and the unrooted unlabelled settings, for $0\leq q_0<q_1\leq 1/2$, the optimal risk of a test to distinguish between $q_0$ and $q_1$ is bounded away from $1/2$, or equivalently, 
    \[\inf _n d_{\tv}(q_0,q_1)>0\, .\]
To prove this, in the rooted unlabelled setting we investigate the size of the two subtrees, $T_n^+$ and $T_n^-$, separated by the root edge, and show that the variable $|T_n^+||T_n^-|$ is a distinguishing statistic. Note that the expectation of $|T_n^+|$ is always equal to $n$ by symmetry, but one might hope that the second moment of $|T_n^+|$, and so the expectation of $|T_n^+||T_n^-|=2n|T_n^+|-|T_n^+|^2$, depends on $q$. This is indeed the case; in particular the expectation is asymptotically of the form $\alpha (q)n^2$, where $\alpha (q)$ is a decreasing function of $q$, while the variance is of order at most $n^4$. 

For the unrooted unlabelled setting, we consider a generalisation of the statistic that was used in the rooted unlabelled setting. Since the root edge is not observable anymore, we consider the sum, over all edges, of the product of the sizes of the two subtrees separated by that edge,
\[
S_n=\sum_{e\in E_n} |T_n^{e_+}||T_n^{e_-}|\, .
\]
One may notice that $S_n$ is equal to the sum of all pairwise distances in the tree, since each edge is counted as many times as the number of paths it belongs to. This variable is also known as the Wiener index (see~\cite{neininger2002wiener} for a detailed analysis on the URT) and it has been used (along with its generalizations) in~\cite{bubeck2017trees,bubeck2015influence,racz2022correlated,curien2015scaling} in the pursuit of distinguishing the seed graph in random growing trees, in the uniform and preferential attachment regimes. We show that the expectation of this variable is asymptotically equivalent to $4n^2\log(n)$, for any $q$, and the dependence on $q$ appears at the order $n^2$. In contrast with the previous section, we now need a much more precise bound on the variance in order to apply the method of distinguishing statistics. This is done by a careful application of the Efron-Stein Inequality. 

\subsection{Open questions}

This work probably leaves much room for improvement and several questions are left open. Let us here mention some of them:
\begin{enumerate}
    \item In the time-labelled setting, we showed that, for $0\leq q\leq 1/2$, the total-variation distance between BCMRT$(q)$ and BCMRT$(1/2)$ tends to $1$ when $1/2-q\gg (\log n)^{-1/4}$, and to $0$ when $1/2-q\ll n^{-1/2}$. Can we reduce the gap between those two rates?
    \item In the time-labelled setting, we showed that for $q\leq 1/50$, clustering is feasible. Is there a threshold for the value of $q$ that separates a regime where clustering is feasible and a regime where it is unfeasible, as it occurs for the Stochastic Block Model? If we had to take a guess on such a value, we would probably say $q=1/4$. Some heuristic considerations sustaining this threshold are given in Remark~\ref{rem:threshold} below.
    \item Is there a polynomial-time clustering algorithm for $q$ small enough? In the Appendix, we draw connections between our clustering procedure and the NP-hard problem of \textit{minimum fair bisection}. In view of that, if polynomial time is possible, we believe the algorithm would have to rely either i) on different model-specific properties of our model rather than the number of monochromatic edges, or ii) on the detailed structure of random colorings of BCMRT trees that would allow to look for fair bisections that are not optimal in terms of cut size but still are inside the margin, or iii) on something different entirely, e.g., spectral techniques or message passing.
    \item In the rooted unlabelled setting, we find that the second moment of $|T_n^+|/n$ depends on $q$. In particular, the limiting distribution depends on $q$. Can we characterize this limiting distribution? For the URT and some of its extensions, the limiting distribution of the fraction of nodes that lie on one given side of the first edge can rather easily be determined using a mapping to P\'olya urns. In the case of the BCMRT, we were not able to find a simple mapping. 
    \item In the unlabelled setting, rooted or even unrooted, is there a test to discriminate two distinct values of $q$ with risk tending to zero? Also, can we estimate $q$ consistently? Put differently, how much can we do without time-labels?
    \item To what extent our results on the BCMRT transfer to the 
 CMRT with $p=1/2$? In the time-labelled setting, our approach heavily relies on the fact that, at each time, one node of each type arrives, which is critical information. For the unlabelled setting, it seems intuitively appealing that the results transfer to the CMRT with $p=1/2$. However, we do not have a rigorous argument for that.
    \end{enumerate}

\begin{remark}\label{rem:threshold}
    Assume that all types are observed except for the ones on the root edge. A simple and natural way to estimate the types on the root edge is to take a majority vote within each of the two subtrees. When does this method perform better than guessing? Let $\Delta_n$ be the difference, at time $n$, between the number of nodes of type $A$ and the number of nodes of type $B$ in the subtree rooted at $(1,A)$. Then, we have $\Delta_1=1$ and a quick investigation shows that the conditional expectation of $\Delta_{n+1}$ given all the history up to time $n$ is equal to
    \[
    \Delta_n\left(1+\frac{1-2q}{n}\right)\, .
    \]
    In particular, the expected difference $\E_q[\Delta_n]$ is of order $n^{1-2q}$ and we believe that, for the majority vote to be positively correlated with the the type of the root, this surplus has to be much larger than $\sqrt{n}$, which corresponds to the typical fluctuations of a Binomial random variable with parameters $n$ and $1/2$. In particular, we expect that if $q>1/4$, then the majority vote does not outperform random guessing. This draws a parallel with the Kesten-Stigum threshold (see~\cite{kesten1966additional,10.1214/aoap/1060202828}), which turns out to be the threshold for clustering in the Stochastic Block Model. Incidentally, $q=1/4$ is also the threshold for reconstruction of the root-bit with majority vote in the broadcasting process on the URT~\cite{addario2022broadcasting}.
\end{remark}

\subsection{Related work}


Popular models such as the stochastic block model (SBM), as well as the Erd\H{o}s-Rényi model from which it is derived, are static models in the sense that they do not describe the growth of the graph. For many networks, such as social networks or the Web, modelling the temporal dynamic of the graph is of large interest. Perhaps the most simple model for growing graphs is the URT, but many variants of this model have been studied: the new node may arrive with possibly more than one edge, or the attachment rule not may not be uniform (\emph{e.g.}, the new node may attach preferentially to a node of high degree), see~\cite{van2009random} for an account on preferential attachment models. To our knowledge, the paper~\cite{bhamidi2022community} contains the first simple model for random recursive trees with multiple communities. 

Various statistical questions have been studied on random recursive trees and more general growing networks. A problem that has gained significant attention in the past few years is \emph{root finding}, that is, constructing confidence sets for the initial node when only the graph structure is observed~\cite{bubeck2017finding, crane2021inference, khim2016confidence, banerjee2022root, briend2023archaeology, crane2021root, banerjee2023degree}.
The problem of testing between different initial configurations (also called \emph{seeds}) and variants of it have been studied in~\cite{bubeck2017trees,bubeck2015influence,curien2015scaling, racz2022correlated}, as well as the problem of recovering the seed~\cite{reddad2019discovery, 10.1214/19-EJP268}.
In~\cite{addario2022broadcasting}, the authors study the broadcasting process on random recursive trees and the possibility of reconstruction of the root bit. The problem of parameter estimation has been considered in~\cite{10.1214/17-EJS1356} and~\cite{tang2020buckley}. In~\cite{casse2023siblings}, the author studies the nearest neighbour tree in the $d$-dimensional sphere and manages to find a distinguishing statistic for $d$. Interestingly, few attempts have been made to make recursive tree models meet the problem of \emph{community detection}. 

The problem of \emph{clustering}, also called community detection, is, given a graph, to find a partition of nodes into strongly connected groups, with few cross-class connections. This problem has a long history, and numerous applications in various domains where network structures arise, from population genetics to image processing and social sciences. From an algorithmic point of view, the theory of clustering reveals some fundamental challenges. The \emph{min-bisection problem} for instance, which asks for a partition of nodes into two groups of equal size with the smallest number of inter-group edges, is known to be NP-hard~\cite{garey1974some}. It quickly became clear that the problem could be made easier on random graphs with a planted bisection. Probably the most studied model is the SBM, in which nodes are divided into two groups of equal size, and then each possible edge is included with probability $p$ if the two endpoints are in the same group, $q$ otherwise, with typically $p>q$. This model turns out to be incredibly rich and exhibits fascinating phase transitions corresponding to the feasibility of clustering with different levels of accuracy. For instance, if $p=a/n$ and $q=b/n$ where $n$ is the number of nodes, it is now known that if $(a-b)^2>2(a+b)$, then it is possible to efficiently cluster in a way correlated with the true partition, while if $(a-b)^2\leq 2(a+b)$, then it is information theoretically impossible, see~\cite{mossel2015reconstruction,massoulie2014community,mossel2018proof}. The same threshold also determines the possibility of estimating the parameters $a$ and $b$~\cite{mossel2015reconstruction}. We refer the reader to~\cite{moore2017computer,abbe2015community,racz2017basic} for references on the subject. Further relevant work can be found in~\cite{hajek2019community} and~\cite{crane2021root}. In~\cite{hajek2019community}, the authors develop a message-passing clustering algorithm for preferential attachment graphs with communities. In~\cite{crane2021root}, 
the authors introduce the \textit{PAPER} model (where part of the edges come from preferential attachment and another part comes from overlaying an Erd\H{o}s-Rényi graph) and propose clustering algorithms for variations of the model with multiple communities.

The purpose of the
present work is to initiate such analyses on random growing trees with equally sized communities. Let us note that the fact that the
observed graph is a tree, while simplifying the analysis, tends to make inference more difficult, since
many observables such as triangles and short cycles, which are typically used for testing and inference
in the SBM (see~\cite{mossel2015reconstruction}), do not occur here. Combining the latter with having communities of equal
size adds another layer of difficulty, see also~\cite{hajek2019community} where this is observed in the case of preferential
attachment graphs with communities.

\section{The degree distribution}\label{sec:degree}

In order to put the model into context with respect to the existing literature, we compute the limiting degree distribution. Moreover, we provide evidence (though not a proof) that this cannot be used to distinguish between different values of $q$.

For an integer $k\geq 1$, let $N_k(n)$ be the number of nodes with degree $k$ in a BCMRT at time $n$. 
\begin{theorem} \label{deg-seq}
In the BCMRT$(q)$ model, for any integer $k\geq 1$ and for any $q\in [0,1]$, we have
 \[\frac{N_k(n)}{2n}\underset{n\to +\infty}{\overset{\P}{\longrightarrow}}\frac{1}{2^k}\, .\]\end{theorem}
  \begin{proof}[Proof of Theorem~\ref{deg-seq}]
To alleviate notation, we will write $N_k(n)$ for both the set of nodes of degree $k$ and the size of this set. For node sets $S_1,S_2$, let $S_1\leftarrow S_2$ denote the event that all nodes in $S_2$ attach with edge to a node in $S_1$. Assume $k\geq 2$. Then $N_k(n+1)$ is equal to
\begin{align} &N_k(n)-\sum_{ \Lambda\in\{A,B\}}\ind_{\{N_k(n)\leftarrow (n+1, \Lambda)\}}+\sum_{u\in N_{k}(n)}\ind_{\{u\leftarrow \{(n+1,A),(n+1,B)\}\}}
\nonumber\\
&+\sum_{ \Lambda\in\{A,B\}}\ind_{\{N_{k-1}(n)\leftarrow (n+1, \Lambda)\}}-2\sum_{u\in N_{k-1}(n)} \ind_{\{u\leftarrow \{(n+1,A),(n+1,B)\}\}}\nonumber\\
&+\sum_{u\in N_{k-2}(n)}\ind_{\{ u\leftarrow \{(n+1,A),(n+1,B)\}\}}\, ,\label{recurrence-degree-first}\end{align}
with $N_0(n)=\emptyset$. Letting $(\cF_n)_{n\geq 1}$ denote the filtration keeping track of the whole history (including types and labels), we have, for $\{ \Lambda,\bar{ \Lambda}\}=\{A,B\}$, 
\[\P\left(N_k(n)\leftarrow (n+1, \Lambda)\,|\,\mathcal{F}_n\right)=(1-q)\frac{N_k^ \Lambda(n)}{n}+q\frac{N_k^{\bar{ \Lambda}}}{n}\, ,\]
therefore the sum of these two probabilities equals $\frac{N_k(n)}{n}$. Moreover, 
\[
\sum_{u\in N_k(n)}\P\left(u\leftarrow \{(n+1,A),(n+1,B)\,|\, \cF_n\right)=\frac{q(1-q)N_k(n)}{n^2}\, \cdot 
\]
Combining~\eqref{recurrence-degree-first} with the latter expressions, 
$\E[N_k(n+1)|\mathcal{F}_n]$ is equal to
\begin{equation}\label{deg-rec}
    N_k(n)\left(1-\frac{1}{n}+\frac{q(1-q)}{n^2}\right)+N_{k-1}(n)\left(\frac{1}{n}-\frac{2q(1-q)}{n^2}\right)+\frac{q(1-q)N_{k-2}(n)}{n^2}~,
\end{equation}
with initial condition $N_k(\lfloor \frac{k}{2}\rfloor )=0$.

Following the same type of reasoning, $\E[N_1(n+1)]$ is equal to
\begin{equation}\label{leaf-rec}
    \E[N_1(n)]\left(1-\frac{1}{n}+\frac{q(1-q)}{n^2}\right)+2~,
\end{equation}
with initial condition $N_1(1)=2$. Using that $q(1-q)\leq \frac{1}{4}$, and solving the induced recurrence, we obtain that 
\begin{align}
    \E[N_1(n+1)]&\leq 2\sum _{i=0}^n\prod _{j=i+1}^n\frac{(j-\frac{1}{2})^2}{j^2}=\frac{2\Gamma(n+\frac{1}{2})^2}{\Gamma(n+1)^2}\sum _{i=0}^{n}\frac{\Gamma(i+1)^2}{\Gamma(i+\frac{1}{2})^2}\, ,\label{jarec}
     \end{align}
where $\Gamma$ is the Gamma function. By log-convexity of $\Gamma$, we have 
    \begin{align*}
        \Gamma\left(n+\frac{1}{2}\right)^2\leq \Gamma(n)\Gamma(n+1)\, ,
    \end{align*}
    and
    \begin{align*}
        \Gamma(i+1)^2\leq \Gamma\left(i+\frac{1}{2} \right)\Gamma\left(i+\frac{3}{2}\right)=\Gamma\left(i+\frac{1}{2}\right)^2\left(i+\frac{1}{2}\right)\, .
    \end{align*}
Therefore,~\eqref{jarec} can be further bounded by
     \begin{align*}
    \frac{2}{n}\sum_{i=0}^n\left(i+\frac{1}{2}\right)\leq n+3\,.
\end{align*}
We can lower-bound $\E[N_1(n+1)]$ by omitting entirely the term $\frac{q(1-q)}{n^2}$ in~\eqref{leaf-rec}, in the following way: 
\[\E[N_1(n+1)]\geq  2\sum_{i=1}^{n-1}\prod_{j=i}^{n-1}\frac{j}{j+1}+2  \geq\frac{2}{n}\sum_{i=1}^ni=n+1\, .\]
Combining both bounds,
\begin{equation}n\leq \E[N_1(n)] \leq n+2\, .\label{expectation-deg}\end{equation}

Returning to~\eqref{deg-rec}, using that for all $k$, $N_k(n)\leq 2n$ and $q(1-q)\leq \frac{1}{4}$, we have that 
\begin{equation}
    \E[N_k(n+1)|\mathcal{F}_n]\leq N_k(n)\left(1-\frac{1}{n}\right)+\frac{N_{k-1}(n)+1}{n}
\end{equation}
and 
\begin{equation}
    \E[N_k(n+1)|\mathcal{F}_n]\geq N_k(n)\left(1-\frac{1}{n}\right)+\frac{N_{k-1}(n)-1}{n}\, .\label{eq2lower}
\end{equation}
Now, one can show by induction in $k$ that there exist constants $\alpha_1(k),\alpha_2(k)$, such that for all $n\geq \lfloor\frac{k}{2}\rfloor$ we have \[\frac{n}{2^{k-1}}+\alpha _2(k)\leq \E[N_{k}(n)]\leq \frac{n}{2^{k-1}}+\alpha _1(k)\, .\] We already showed this for $N_1(n)$. 
Assuming that the claim holds for $k$, we have 
\begin{align*}
     \E[N_{k+1}(n+1)]\leq \sum_{i=\lfloor \frac{k}{2}\rfloor}^n\frac{2^{-k+1}i+\alpha_1(k)+1}{i}\prod _{j=i+1}^n\frac{j-1}{j}\leq \frac{n}{2^k}+\alpha_1(k)+1+\frac{1}{2^k}\, .
\end{align*}
The lower bound can be shown in a similar way, using~\eqref{eq2lower}. 

Convergence in probability follows by applying McDiarmid's inequality~\citep{mcdiarmid1989method} (also referred to as \textit{bounded differences inequality}).
    \end{proof}
    
        The limiting degree distribution of the BCMRT coincides with the one of the uniform random recursive tree (see~\citep{janson2005asymptotic} for an account of the behaviour of $N_k(n)$ in that model, starting from~\citep{na1970distribution} where it was first shown that $\E[N_k(n)]/n\rightarrow 2^{-k}$). Moreover, by the proof of Theorem~\ref{deg-seq}, we have that for any $k\geq 1$ and $q\neq q'$, $|\E _q[N_k(n)]-\E _{q'}[N_k(n)]|=\mathcal{O}(1)$ while we suspect the variance of $N_k(n)$ to be of order $n$. Hence we don't expect that $N_k(n)$ can be used as distinguishing statistics. In order to further solidify our intuition, we provide a proof that the variance of $N_1(n)$ is of linear order. 
    \begin{theorem}\label{variance1}
    For all $q\in [0,1[$, we have $\Var \left(N_1(n)\right)=\Omega \left(n\right)$.
    \end{theorem}

     \begin{proof}[Proof of Theorem~\ref{variance1}] Let $(\cF_n)_{n\geq 1}$ denote the filtration keeping track of the whole history of the tree, including types and time-labels and write $\E_i$ to denote the expectation conditional on $\cF_i$. We use the martingale decomposition
        \begin{equation}
            \Var\left(N_1(n)\right)=\sum _{i=1}^{n-1}\E[(\E_{i+1}[N_1(n)]-\E_{i}[N_1(n)])^2]\, .\label{var-dec}
        \end{equation}
        Using the tower property for expectation and~\eqref{leaf-rec}, we obtain that $\E_{i}[N_1(n)]$ is equal to 
        \begin{equation}
            N_1(i)A(i)+B(i)~,\label{diff2}
        \end{equation}
        where 
        \[A(i)=\prod_{k=i}^{n-1} \left(1-\frac{1}{k}+\frac{q(1-q)}{k^2}\right)\quad \text{ and }\quad B(i)=2\sum _{j=i}^{n-1}\prod _{k=j+1}^{n-1}\left(1-\frac{1}{k}+\frac{q(1-q)}{k^2}\right)\, .\]
        We have
          \begin{align}
                    A(i+1)-A(i)&=\prod_{k=i+1}^{n-1}\left(1-\frac{1}{k}+\frac{q(1-q)}{k^2}\right)\left(1-\left(1-\frac{1}{i}+\frac{q(1-q)}{i^2}\right)\right)\nonumber\\
                    &\geq\left(\frac{1}{i}-\frac{1}{4i^2}\right)\prod_{k=i+1}^{n-1}\frac{k-1}{k}=\left(1-\frac{1}{4i}\right)\frac{1}{n-1}\geq \frac{1}{2n}\, .\label{lbound}
        \end{align}
        Moreover, 
      \begin{equation}B(i+1)-B(i)=-2A(i+1)\, . \label{betaeq}\end{equation}
Therefore, on the event that $N_1(i+1)=N_1(i)+2$ we have  
\begin{align}
 (\E_{i+1}[N_1(n)]-\E_{i}[N_1(n)])^2 &\geq  ( N_1(i)(A(i+1)-A(i)))^2\geq N_1(i)^2\frac{1}{4n^2}\, ,\label{new-bound}
\end{align}
 where for the first inequality we use ~\eqref{diff2},\eqref{betaeq}, and for the second one we use~\eqref{lbound}.

To compute the probability of this event given $\mathcal{F}_i$, write $\bar{N}^ \Lambda_1(i)$ for the nodes of type $ \Lambda$ at time $i$ that are not leaves and the overline on $ \Lambda$ denotes its complement
in the set $\{A, B\}$. We have that $\P\left(N_1(i+1)=N_1(i)+2\,|\,\mathcal{F}_i\right)$ is equal to
        \begin{align}
\frac{1}{i^2}\left((q^2+(1-q)^2)\bar{N}_1^ \Lambda(i)\bar{N}_1^{\bar{ \Lambda}}(i)+q(1-q)\left(\bar{N}_1^ \Lambda(i)^2+\bar{N}_1^{\bar{ \Lambda}}(i)^2\right)\right)\geq  \frac{q(1-q)}{i^2} \bar{N}^2(i)\, .\label{max-prob1}
        \end{align}
Hence, 
\begin{align}\E\left[\P\left(N_1(i+1)=N_1(i)+2\,|\,\mathcal{F}_i\right)N_1(i)^2\right]&\geq  \frac{q(1-q)}{i^2}\E\left[N_1(i)^2\bar{N}_1(i)^2\right]\, ,\nonumber 
\end{align}
which is further bounded by $q(1-q)i^{-2}\E\left[N_1(i)\bar{N}_1(i)\right]^2$, using Jensen's inequality. But we have that $\E\left[N_1(i)\bar{N}_1(i)\right]=\Omega (i^2)$, since the following recurrence holds by conditioning on $\mathcal{F}_i$ and then taking expectation on both sides:
\[
\E\left[N_1(i+1)\bar{N}_1(i+1)\right]\,\geq\,   \E\left[N_1(i)\bar{N}_1(i)\right]+\frac{1}{i}\E\left[N_1(i)^2\right]\, \geq\, \E\left[N_1(i)\bar{N}_1(i)\right]+i\, ,
\]
where the last inequality follows by Jensen's inequality and~\eqref{expectation-deg}.

Recalling~\eqref{var-dec} and combining with the previous relations, we conclude that
    \begin{align*}
  \Var\left(N_1(n)\right)&\geq  \frac{q(1-q)}{4n^2} \sum_{i=1}^{n-1}i^{-2} \E\left[N_1(i)\bar{N}_1(i)\right]^2=\Omega (n)
\end{align*}
when $q\neq 0$. If $q=0$, the tree consists of two independent monochromatic subtrees that grow as uniformly random recursive trees. In that case, it is already known that the variance grows linearly; see for instance~\citep{janson2005asymptotic} or~\citep{drmota2009random}.
 \end{proof}

\section{Inference, testing and clustering in time-labelled BCMRT}
\label{sec:bcmrt}

\subsection{Inference}

In the time-labelled setting, one may define an estimator for $q$ as follows. 

For $k\geq 2$, let $X_k$ be the indicator that $(k,A)$ and $(k,B)$ attach to the same node, and define 
\begin{equation}\label{eq:def-Z_n}
Z_n=\sum_{k=2}^n X_k\, .
\end{equation}
Since the sequence $(X_k)_{k=2}^n$ is independent with $X_k$ distributed as a Bernoulli random variable with parameter $\mu_k=\frac{2q(1-q)}{k-1}$, we have
\[
\E_q[Z_n]\,=\,\sum_{k=2}^n \mu_k\,\underset{n\to\infty}{\sim} \,2q(1-q)\log n\quad\text{ and }\quad \Var_q(Z_n)\,=\, \sum_{k=2}^n \mu_k(1-\mu_k)\,\underset{n\to\infty}{\sim} \,\E_q[Z_n]\, .
\]
By Chebyshev's inequality, we see that
\begin{equation}\label{eq:conv-prob-Z}
\frac{Z_n}{2\log n} \underset{ n\to \infty}{\overset{\P}{\;\longrightarrow\;}}q(1-q) \, .
\end{equation}
Inverting the function $q\mapsto q(1-q)$ on $[0,1/2]$ leads us to consider the estimator
\begin{equation}\label{eq:def-hat-q}
\hat{q}_n\,=\, \varphi\left(\frac{Z_n}{2\log n} \right)\, ,
\end{equation}
with $\varphi(x)=\frac{1}{2}\left(1-\sqrt{(1-4x)_+}\right)$ for $x\geq 0$. Here, $(u)_+$ stands for the positive part of $u$. It makes $\hat{q}_n$ well-defined, even when $\frac{Z_n}{2\log n}>\frac{1}{4}$.

In the next theorem, we show that $\hat{q}_n$ is a consistent estimator of $q$ and determine its speed of convergence. Interestingly, the approximation deteriorates as $q$ approaches $1/2$.

\begin{theorem}\label{thm:estim-labelled-BCMRT}
Let $\hat{q}_n$ be defined as in~\eqref{eq:def-hat-q}. Then, for all $0\leq q\leq 1/2$, if the tree is a BCMRT$(q)$, 
    \[
    \hat{q}_n\underset{ n\to \infty}{\overset{\P}{\;\longrightarrow\;}} q\, .
    \]
Moreover,
    \begin{itemize}
        \item if $q=0$, then $\hat{q}_n$ is a Dirac mass at $0$,
        \item if $q\in (0,1/2)$, then
        \[
\left(\frac{2\log n}{q(1-q)}\right)^{1/2}\left(\hat{q}_n-q\right) \, \underset{n\to \infty}\hookrightarrow\, \frac{1}{1-2q}\cN(0,1)\, ,
\]
\item if $q=1/2$, then 
\[
\left(\frac{\log n}{2}\right)^{1/4}\left(\hat{q}_n-\frac{1}{2}\right)\, \underset{n\to \infty}\hookrightarrow\, -\frac{1}{2}\sqrt{(\cN(0,1))_+}\, .
\]
    \end{itemize}
\end{theorem}

\begin{proof}[Proof of Theorem~\ref{thm:estim-labelled-BCMRT}]
The first statement follows simply from~\eqref{eq:conv-prob-Z} and the continuity of $\varphi$. As for the fluctuations, we first note that when $q=0$, the variable $\hat{q}_n$ is simply a Dirac mass on $0$. Now for $q\in (0,1/2)$, using that $|X_k-\mu_k|\leq 1$, and Chebyshev's inequality, we have, for all $\varepsilon>0$,
\begin{align*}
&\frac{1}{\Var_q(Z_n)}\sum_{k=2}^n \E\left[(X_k-\mu_k)^2\ind_{\left\{|X_k-\mu_k|>\varepsilon\sqrt{\Var_q(Z_n)}\right\}}\right]\\
&\hspace{1cm}\leq\, \frac{1}{\Var_q(Z_n)}\sum_{k=2}^n \frac{\Var_q(X_k)}{\varepsilon^2\Var_q(Z_n)}\,=\,\frac{1}{\varepsilon^2\Var_q(Z_n)}\, =\, o(1)\, .
\end{align*}
Therefore, Lindeberg's condition is satisfied and we have
\begin{equation}\label{eq:lindeberg}
\sqrt{\frac{2\log n}{q(1-q)}}\left(\frac{Z_n}{2\log n}-q(1-q)\right) \, \underset{n\to \infty}\hookrightarrow\, \cN(0,1)\, ,
\end{equation}
where $\hookrightarrow$ stands for convergence in distribution and $\cN(0,1)$ is a standard Gaussian random variable. Since the function $\varphi$ is differentiable on $(0,1/4)$ with $\varphi'(x)=(1-4x)^{-1/2}$, we have $\varphi'(q(1-q))=(1-2q)^{-1}$ for $q\in (0,1/2)$, and a simple Taylor expansion gives
\[
\sqrt{\frac{2\log n}{q(1-q)}}\left(\hat{q}_n-q\right) \, \underset{n\to \infty}\hookrightarrow\, \frac{1}{1-2q}\cN(0,1)\, .
\]
For $q=1/2$, we have
\[
\hat{q}_n-\frac{1}{2}=-\frac{1}{2}\sqrt{\left(1-\frac{2Z_n}{\log n}\right)_+}\, ,
\]
but by~\eqref{eq:lindeberg}, we know that $\sqrt{\frac{\log n}{2}}\left(1-\frac{2Z_n}{\log n}\right)$ converges in distribution to $\cN(0,1)$, hence
\[
\left(\frac{\log n}{2}\right)^{1/4}\left(\hat{q}_n-\frac{1}{2}\right)\, \underset{n\to \infty}\hookrightarrow\, -\frac{1}{2}\sqrt{(\cN(0,1))_+}\, .
\]
\end{proof}

\subsection{Testing}

We now move on to the following testing problem:
\[
H_0\,:\; q=q_0\qquad H_1\,:\; q=q_1\, ,
\]
with $0\leq q_0<q_1\leq 1/2$. A test $\varphi$ of $H_0$ versus $H_1$ is a function from the set of time-labelled trees that can occur as realisations of BCMRT to $\{0,1\}$. The risk of a test $\varphi$ is defined by
\[
\bR_{q_0,q_1}(\varphi)\,=\,\frac{1}{2}\left(\P_{q_0}(\varphi=1)+\P_{q_1}(\varphi=0)\right)\, ,
\]
and the optimal risk is
\[
\bR^\star_{q_0,q_1}=\,\inf_{\varphi} \bR_{q_0,q_1}(\varphi)\, ,
\]
where the infimum is taken over all tests $\varphi$. 

In the following theorem, we allow the parameter $q$ to depend on $n$, which is highlighted by the notation $q^{(n)}$. Let us stress that this does not mean that the parameter varies with time as the tree grows, but just that it may depend on the total size of the tree.

\begin{theorem}\label{thm:testing-labelled-BCMRT}
Assume that one observes a time-labelled BCMRT. Then, the following holds:
\begin{enumerate}
    \item[$(i)$] If $0\leq q_0^{(n)}<q_1^{(n)}\leq 1/2$ and 
    \[
    \left(q_1^{(n)}-q_0^{(n)}\right)^2\left(1-q_0^{(n)}-q_1^{(n)}\right)^2\gg \frac{q_1^{(n)}}{\log n}\, ,
    \]
    then
    \[
    \bR^\star_{q_0^{(n)},q_1^{(n)}}\underset{ n\to \infty}{\;\longrightarrow\;} 0\, .
    \]
    In particular, if $q_1^{(n)}=1/2$ and $\frac{1}{2}-q_0^{(n)}\gg \frac{1}{(\log n)^{1/4}}$, then $\bR^\star_{q_0^{(n)},1/2}\underset{ n\to \infty}{\;\longrightarrow\;} 0$.
    \item[$(ii)$]  If $\frac{1}{2}-q_0^{(n)}\ll \frac{1}{\sqrt{n}}$, then
    \[
    \bR^\star_{q_0^{(n)},1/2}\underset{ n\to \infty}{\;\longrightarrow\;} \frac{1}{2}\, .
    \]
\end{enumerate}
\end{theorem}

\begin{proof}[Proof of Theorem~\ref{thm:testing-labelled-BCMRT}]
To alleviate notation, we write $q_0$ and $q_1$ instead of $q_0^{(n)}$ and $q_1^{(n)}$.
To prove the first statement $(i)$, let $0\leq q_0<q_1 \leq 1/2$, and consider the test $\varphi_n$ given by
\[
\varphi_n=\ind_{\left\{ Z_n >\frac{\E_{q_0}[Z_n]+\E_{q_1}[Z_n]}{2}\right\}}\, ,
\]
with $Z_n$ defined in~\eqref{eq:def-Z_n}. By Chebyshev's Inequality and the fact that $\Var_q(Z_n)\leq \E_q[Z_n]$, we have
\begin{align*}
    \bR_{q_0,q_1}(\varphi_n) &=\,\frac{1}{2} \P_{q_0}\left(Z_n-\E_{q_0}[Z_n]>\frac{\E_{q_1}[Z_n]-\E_{q_0}[Z_n]}{2}\right) \\
    &\hspace{2.5cm}+\frac{1}{2}\P_{q_1}\left(Z_n-\E_{q_1}[Z_n]\leq \frac{\E_{q_1}[Z_n]-\E_{q_0}[Z_n]}{2}\right)\\
&\leq\, \frac{2(\E_{q_0}(Z_n)+\E_{q_1}(Z_n))}{(\E_{q_1}[Z_n]-\E_{q_0}[Z_n])^2}\\
    &=\, \frac{q_0(1-q_0)+q_1(1-q_1)}{\left(q_1(1-q_1)-q_0(1-q_0)\right)^2H_{n-1}}\\
    &\leq\, \frac{2q_1}{(q_1-q_0)^2(1-q_1-q_0)^2H_{n-1}}\, ,
\end{align*}
where $H_{n-1}=\sum_{k=1}^{n-1} \frac{1}{k}$. We see that, if $(q_1-q_0)^2(1-q_0-q_1)^2\gg \frac{q_1}{\log n}$, then $\bR_{q_0,q_1}(\varphi_n) \underset{ n\to \infty}{\;\longrightarrow\;} 0$. Note that when $q_1=1/2$, then the condition becomes $1/2-q_0\gg (\log n)^{-1/4}$, whereas when $q_0=0$, it becomes $q_1\gg (\log n)^{-1}$, and when $q_0$ and $q_1$ are bounded away from $0$ and $1/2$, it reads $q_1-q_0\gg (\log n)^{-1/2}$. Finally, when $q_0,q_1\rightarrow 0$, it amounts to $q_1-q_0\gg q_1^{1/2}(\log n)^{-1/2}$.

Let us now move to the proof of statement $(ii)$. Let $T$ be a time-labelled balanced recursive tree of size $2n$. Let $\cC$ be the set of valid colorings of $T$, i.e., colorings with colors $A$ and $B$ such that each pair of nodes with the same time label is colored with both $A$ and $B$. We consider that the two different nodes labelled $i$ get an arbitrary ordering. In this way, a coloring is an $n$-tuple of zeros and ones, where ones correspond to ordered pairs colored $(A,B)$, and zeros to ordered pairs colored $(B,A)$. For $\sigma\in\cC$, we denote by $M_T^\sigma$ the number of monochromatic edges in the tree $T$ colored with $\sigma$. 

Each coloring in $\mathcal{C}$ corresponds to a realisation of BCMRT$(q)$ and the probability of the realisation is equal to $(1-q)^{M_T^\sigma}q^{2(n-1)-M_T^\sigma}\prod_{k=2}^n(k-1)^{-2}$. Notice that different colorings might correspond to the same realisation. For instance, consider the two colorings $(1,0)$ and $(1,1)$ applied to a tree of four nodes where the two nodes of label 2 are both connected to the same node of label 1. In order not to double count realisations, we need to investigate how many different colorings correspond to the same realisation. To that end, assume two such different colorings $c_1,c_2$. Then there exists a pair of nodes labelled $i$, such that $c_1,c_2$ differ in the $i$-th position. Pick the smallest such $i$. The descendant sets of these two nodes are the same and they all have different color in $c_1$ than in $c_2$, otherwise the realisation is not the same. Moreover, by minimality of $i$, the two nodes labelled $i$ must connect to the same node (except if $i=1$), otherwise the realisation is not the same. Erase the subtrees hanging from the two nodes labelled $i$ and continue recursively until there is no such pair of nodes.
Conversely, assume an arbitrary coloring and find, if they exist, two nodes with the same label that connect to the same node (or they are the root nodes) and their descendants form the same tree, as time-labelled rooted trees, and flip all the colors in these two subtrees. Then we end up with two different colorings that correspond to the same realisation. Let $\rm{ec}(T)$ be the number of such pairs of nodes. We conclude that, given a coloring $c$, the number of colorings that give the same realisation as $c$ is equal to $2^{\rm{ec}(T)}$, which depends only on the time-labelled tree and not on $c$. Then, the probability $P_q(T)$ to generate $T$ under BCMRT$(q)$ is
\[
P_q(T)=\sum_{\sigma\in\cC} \frac{1}{2^{\rm{ec}(T)}}(1-q)^{M_T^\sigma}q^{2(n-1)-M_T^\sigma}\prod_{k=2}^n \frac{1}{(k-1)^2}\, .
\]

In what follows, we write for simplicity $P$ instead of $P_{1/2}$.
The possibility of distinguishing $P_q$ from $P$ is determined by the total-variation distance between the two probability distributions defined by
\[
\|P_q-P\|_\tv=\frac{1}{2}\sum_T |P_q(T)-P(T)|\, .
\]
More precisely, it is a well-known fact (see for instance~\cite{lugosi2017lectures}, page 5) that
\[
\bR^\star_{q,1/2}=\frac{1}{2}\left( 1-\|P_q-P\|_\tv\right)\, .
\]
Thus our goal reduces to showing that, if $\frac{1}{2}-q\ll n^{-1/2}$, then $\|P_q-P\|_\tv \underset{ n\to \infty}{\;\longrightarrow\;}0$. By Cauchy-Schwarz Inequality, we have
\[
\|P_q-P\|_\tv\,\leq\,\frac{1}{2}\sqrt{\sum_T \frac{P_q(T)^2}{P(T)} -1}\, .
\]
Note that the latter expression is also known as $\chi ^2$-divergence. Moreover, 
\begin{align*}
\sum_T \frac{P_q(T)^2}{P(T)} \,&=\,(2q)^{2(n-1)}\sum_T P_q(T) \sum_{\sigma\in\cC}\frac{1}{2^{n}} \left(\frac{1-q}{q}\right)^{M_T^\sigma}\\
&=\, (2q)^{2(n-1)}\sum_{\sigma\in\cC}\frac{1}{2^{n}}\E\left[ \left(\frac{1-q}{q}\right)^{M_\cT^{\sigma}}\right]\, ,
\end{align*}
where the expectation is over $\cT\sim P_q$, and the $2^n$ in the denominator is the number of all possible colorings. Next, note that for any coloring $\sigma$, the random variable $M_\cT^\sigma$ is stochastically dominated by a sum of $2(n-1)$ independent Bernoulli random variables with parameter $1-q$. Indeed, the probability for each new node to attach to a node that has the same color in $\sigma$ is always less than $1-q$ (if $\sigma$ corresponds to the true coloring, then there is equality).
Hence, for all $\sigma\in\cC$, we have
\[
\E\left[ \left(\frac{1-q}{q}\right)^{M_\cT^{\sigma}}\right]\leq \left(\frac{(1-q)^2}{q}+q\right)^{2(n-1)}\, .
\]
For $q=\frac{1-\varepsilon_n}{2}$ with $\varepsilon_n\in (0,1)$, we obtain
\[
\sum_T \frac{P_q(T)^2}{P(T)}\leq \left(\frac{(1+\varepsilon_n)^2}{2}+\frac{(1-\varepsilon_n)^2}{2}\right)^{2(n-1)}=(1+\varepsilon_n^2)^{2(n-1)}\, ,
\]
which tends to $1$ as soon as $\varepsilon_n\ll n^{-1/2}$, establishing Theorem~\ref{thm:testing-labelled-BCMRT}, $(ii)$.
\end{proof}

\subsection{Clustering}

Let $T$ be a time-labelled BCMRT of size $2n$, and let $\pi$ be the true coloring used to generate $T$. First note that the number of monochromatic edges in $T$ with respect to $\pi$ can be written
\[
M_T^\pi=\sum_{i=2}^{n}(\xi_i+\xi'_i)\, ,
\]
where $\xi_i,\xi'_i$ are i.i.d.\ Bernoulli random variables with parameter $1-q$. In particular, $\E\left[M_T^\pi\right]=2(1-q)(n-1)$ and, by Hoeffding inequality, we have
\[
\P\left(|M_T^\pi-2(1-q)(n-1)|\geq n^{2/3}\right)=e^{-\Omega(n^{1/3})}\, .
\]
We will show that for $q$ small enough, with probability tending to $1$ as $n\to +\infty$, any coloring of $T$ which has a number of monochromatic edges larger than $s_n=2(1-q)(n-1)-n^{2/3}
$ must be correlated with $\pi$. This gives an exponential time clustering algorithm which performs better than chance: simply go through all colorings and output the first one for which $M_T^\sigma\geq s_n$. Let us mention that this approach was used by J.~Banks, C.~Moore, J.~Neeman, and P.~Netrapalli in~\cite{banks2016information}, in the context of the Stochastic Block Model.

For a given coloring $\sigma$ of $T$, let $m(\sigma)$ be the number of node pairs (with the same time label) which are correctly colored:
\[
m(\sigma)=\sum_{i=1}^{n}\ind_{\sigma_i=\pi_i}\, .
\]
\begin{theorem}\label{thm:clustering}
  There exists an absolute positive constant $\varepsilon>0$, such that, for $q$ small enough, the probability that there exists $\sigma$ satisfying both 
    \[
    \left|m(\sigma)-\frac{n}{2}\right|\leq \varepsilon n \qquad \text{ and }\qquad M_T^\sigma\geq 2(1-q)(n-1)-n^{2/3}
    \]
    tends to $0$ as $n\to +\infty$ (in fact, the proof guarantees at least the range $q\leq 1/50$).
\end{theorem}
\begin{proof}[Proof of Proposition~\ref{thm:clustering}]
Let $\sigma$ be a given coloring with $1\leq m\leq n$ correctly colored pairs and assume that $\left|m-\frac{n}{2}\right|\leq \varepsilon n$ for some $\varepsilon>0$ that will be specified later. Without loss of generality, we further assume that
\[
 \frac{n}{2}\leq m\leq \frac{n}{2}+ \varepsilon n \, ,
\]
since the other situation can be treated similarly. Let $(c_1,\dots,c_{n})\in\{0,1\}^{n}$ be the binary vector indicating the location of the correctly colored pairs and define
\[
\omega_i=\frac{\sum_{j=1}^ic_j}{i}\, \cdot 
\]
Here, we view $\sigma$ and $\pi$ as fixed assignment of colors for each node pair, defined before the tree is generated. We then generate $T$ according to the coloring $\pi$ and look at the number of monochromatic edges it induces for $\sigma$. For all $\lambda\geq 0$, we have
\[
    \ln\E\left[ e^{\lambda M_T^\sigma}\right]= 2\sum_{i=2}^{n}\left\{c_i\ln(p_ie^\lambda+1-p_i)+(1-c_i)\ln((1-p_i)e^\lambda+p_i)\right\}\, ,
\]
where
\[
p_i=(1-q)\omega_{i-1}+q(1-\omega_{i-1})\, .
\]
Indeed, the number of monochromatic edges added at step $i$ is a sum of two independent Bernoulli random variables with parameter $p_i$ when $c_i=1$, and $1-p_i$ when $c_i=0$. By Jensen inequality, we have
\begin{equation}\label{eq:bound-laplace}
    \ln\E\left[ e^{\lambda M_T^\sigma}\right]\leq 2(n-1)\ln(pe^\lambda+1-p)\, ,
\end{equation}
where 
\begin{align*}
p& =\frac{1}{n-1}\sum_{i=2}^{n}\left\{c_ip_i+(1-c_i)(1-p_i)\right\}\\
&= \frac{1-q}{n-1}\sum_{i=2}^{n}\left\{c_i\omega_{i-1}+(1-c_i)(1-\omega_{i-1})\right\}+\frac{q}{n-1}\sum_{i=2}^{n}\left\{c_i(1-\omega_{i-1})+(1-c_i)\omega_{i-1}\right\}\, .
\end{align*}
We see that the value of $p$ depends on the configuration of ones in the vector $(c_1,\dots,c_{n})$. Since $p$ is a weighted average between $1-q$ and $q$, and since $1-q\geq q$, the configuration that maximizes the value of $p$ is the one that maximizes the weight put on $1-q$, i.e.,
\[
\frac{1}{n-1}\sum_{i=2}^{n}\left\{c_i\omega_{i-1}+(1-c_i)(1-\omega_{i-1})\right\}\, ,
\]
over all vectors $c\in\{0,1\}^n$ with $m$ ones and $n-m$ zeros. We claim that this quantity is maximized at the vector $c$ with all the $m$ ones at the beginning, followed by $n-m$ zeros. This is the content of Lemma~\ref{lem:worst-case-p}, that can be found at the end of this section. Hence, $p$ is always less than
 \[
\frac{1-q}{n-1}\left(m-1+\sum_{i=m+1}^{n}\left(1-\frac{m}{i-1}\right)\right)+\frac{q}{n-1}\sum_{i=m+1}^{n}\frac{m}{i-1}\, .
\]
Next, note that
\[
\sum_{i=m+1}^{n}\left(1-\frac{m}{i-1}\right)=\sum_{i=1}^{n-m}\frac{i-1}{i+m-1}\leq \sum_{i=1}^{n-m}\frac{i}{i+m}\, \cdot 
\]
Using that $m\geq n-m$ and that $n-m$ tends to $+\infty$ with $n$, we have
\[
\frac{1}{n-m}\sum_{i=1}^{n-m}\frac{i}{i+m}\leq \frac{1}{n-m}\sum_{i=1}^{n-m}\frac{i}{i +n-m}\;\underset{n\to +\infty}{\longrightarrow}\; \int_0^1 \frac{x}{x+1}dx=1-\ln(2)\, .
\]
Hence, using that $\frac{m-1}{n-1}\leq \frac{1}{2}+\varepsilon$ and that $\limsup_{n\to\infty} \frac{n-m}{n-1}\leq \frac{1}{2}$, we have for $n$ large enough,
\[
p\leq (1-q)\left(\frac{1}{2}+2\varepsilon+\frac{1}{2}(1-\ln(2))\right)+\frac{q\ln(2)}{2}\, ,
\]
where the additional $\varepsilon$ compensates for the asymptotics. Rearranging, we get
\[
p\leq \bar{p}=1+2\varepsilon-\frac{\ln(2)}{2}-q(1-\ln(2))\, .
\]
Coming back to~\eqref{eq:bound-laplace} and using a Chernoff bound, we have
\[
\P\left(M_T^\sigma\geq 2(1-q)(n-1)-n^{2/3}\right)\leq e^{-2(n-1)\sup_{\lambda\geq 0} \left\{(1-q)\lambda-\ln(\bar{p}e^\lambda+1-\bar{p})-\frac{\lambda n^{2/3}}{2(n-1)}\right\}}\, .
\]
Taking
\[
\lambda=\ln\left(\frac{(1-\bar{p})(1-q)}{\bar{p}q}\right)\, ,
\]
we find 
\[
\P\left(M_T^\sigma\geq 2(1-q)(n-1)-n^{2/3}\right)\leq  e^{-2(n-1)\left((1-q)\ln\left(\frac{(1-\bar{p})(1-q)}{\bar{p}q}\right)-\ln\left(\frac{1-\bar{p}}{q}\right)+o(1)\right)}\, .
\]
Finally, taking a union bound and using that the number of colorings $\sigma$ with $\left|m(\sigma)-\frac{n}{2}\right|\leq \varepsilon n $ is less than $2^{n}$, we see that the proof can be concluded if we can find $q$ and $\varepsilon$ such that
\[
(1-q)\ln\left(\frac{(1-\bar{p})(1-q)}{\bar{p}q}\right)-\ln\left(\frac{1-\bar{p}}{q}\right)-\frac{\ln(2)}{2}>0\, .
\]
This holds for $\varepsilon=10^{-4}$ and $q\leq \frac{1}{50}$.
\end{proof}

\begin{lemma}\label{lem:worst-case-p}
    For $1\leq m\leq n$, let $\cX_m$ be the subset of vectors $x\in \{0,1\}^n$ with $m$ ones. For $x\in\cX_m$, let
    \[
    F(x)=\sum_{i=2}^{n} x_i\omega_{i-1}+(1-x_i)(1-\omega_{i-1})\, ,
    \]
    where $\omega_i=\frac{1}{i}\sum_{j=1}^i x_j$. Then, if $m\geq n/2$, the maximum of $F$ over $\cX_m$ is attained for the vector consisting of $m$ ones at the beginning, followed by $n-m$ zeros, while if $m\leq n/2$, it is attained for the vector with zeros first. 
\end{lemma}

\begin{proof}[Proof of Lemma~\ref{lem:worst-case-p}]
Suppose that $m\geq n/2$ (the case $m\leq n/2$ is completely symmetric). We will show that, starting from any configuration $x\in\cX_m$, there is a sequence of moves that can only increase the value of $F$ and that arrives at the configuration $(1,\dots,1,0,\dots,0)$. Assume that, for $a\in\{0,1\}$, and for $i\geq 1$, $j\geq 1$ such that $i+j+1\leq n$, the current configuration $x$ starts with $i$ bits equal to $a$, followed by $j$ bits equal to $b=1-a$, followed by one bit equal to $a$. The remainder of the vector is arbitrary, provided the condition $x\in\cX_m$ is satisfied:
\[
x=(\underbrace{a,\dots,a}_{i},\underbrace{b,\dots,b}_{j},a,x_{i+j+2},\dots,x_n)\, .
\]
We claim that if $i+1\geq j$, then the value of $F$ is increased by moving the bit $a$ from position $i+j+1$ to position $1$, while if $j\geq i+1$, then the value of $F$ is increased by moving the blocks of $b$ at the beginning. Let us fist consider the case $i+1\geq j$, and let
\[
x'=(\underbrace{a,a,\dots,a}_{i+1},\underbrace{b,\dots,b}_{j},x_{i+j+2},\dots,x_n)\, .
\]
Then,
\begin{align*}
    F(x')-F(x)&=\left\{i+\sum_{k=1}^{j-1}\frac{k}{i+k+1}\right\}-\left\{i-1+\sum_{k=1}^{j-1}\frac{k}{i+k}+\frac{i}{i+j}\right\}\\
    &= \frac{j}{i+j}-\sum_{k=1}^{j-1}\frac{k}{(i+k)(i+k+1)}\, \cdot
\end{align*}
Noting that the function $x\mapsto \frac{x}{(i+x)(i+x+1)}$ is increasing for $0\leq x\leq i$ and using that $j-1\leq i$, we have
\[
    F(x')-F(x)\geq  \frac{j}{i+j}-\frac{(j-1)^2}{(i+j-1)(i+j)}\geq 0\, \cdot
\]
Let us now turn to the case $j\geq i+1$, and let
\[
x''=(\underbrace{b,\dots,b}_{j},\underbrace{a,\dots,a,a}_{i+1},x_{i+j+2},\dots,x_n)\, .
\]
We have
\begin{align*}
    F(x'')-F(x)&=\left\{j-1+\sum_{k=1}^{i}\frac{k}{j+k}\right\}-\left\{i-1+\sum_{k=1}^{j-1}\frac{k}{i+k}+\frac{i}{i+j}\right\}\\
    &= \frac{j}{i+j}-\sum_{k=1}^i\frac{j}{j+k}+\sum_{k=1}^{j-1}\frac{i}{i+k}\\      
    &\geq \frac{j}{i+j}-\sum_{k=1}^i\frac{k(j-i)}{(i+k)(j+k)}\, \cdot 
\end{align*}
Noting that the function $x\mapsto \frac{x}{(i+x)(j+x)}$ is increasing for $0\leq x\leq \sqrt{ij}$ and using that $j\geq i+1$, we have
\[
    F(x')-F(x)\geq  \frac{j}{i+j}-\frac{i(j-i)}{2i(j+i)}\geq 0\, \cdot
\]
To conclude, observe that when $m> n/2$, iterating those transitions (choosing $x'$ or $x''$ depending on the current situation) will end at $(1,\dots,1,0,\dots,0)$. If $m=n/2$, it might also end at $(0,\dots,0,1,\dots,1)$ but then for both vectors the value of $F$ is the same.
\end{proof}

\begin{remark}
    The task that our clustering algorithm performs is closely related to the task of finding a fair bisection of the nodes such that the number of cut edges is minimised (by \textit{fair} it is meant that neither of the parts contains two nodes with the same label). To see that, notice that a valid coloring of a BCMRT tree with a large number of monochromatic edges corresponds in a natural way to a fair bisection of the nodes such that the number of cut edges is small. Therefore, with probability tending to~$1$, finding a minimum fair bisection of the nodes satisfies the condition of our clustering algorithm (for which the cut size does not need to be optimal). The problem of minimum fair bisection was proved to be NP-hard in~\cite{casel_et_al:LIPIcs.FORC.2023.9} for general trees with two nodes for each label, by using a reduction from the \textit{vertex cover} problem in cubic graphs, that was developed in~\cite{bonsma2006complexity}. In the Appendix, we show that the minimum fair bisection remains NP-hard for recursive trees.
\end{remark}

\section{Testing in rooted unlabelled BCMRT}\label{sec:rooted}

We now remove the time labels, except for the root-edge, i.e., we observe an unlabelled tree on $2n$ nodes, with one distinguished edge connecting the two initial nodes. In this context, we establish the following result.

\begin{theorem}\label{thm:rooted-unlabelled-BCMRT}
Assume that one observes a rooted unlabelled {\rm BCMRT}. Then, for all $0\leq q_0<q_1\leq 1/2$, we have
\[
\sup_{n\geq 2} \bR^\star_{q_0,q_1} <\frac{1}{2}\, .
\]
\end{theorem}
The proof of Theorem~\ref{thm:rooted-unlabelled-BCMRT} (as well as the one of Theorem~\ref{thm:unrooted-unlabelled-BCMRT} in the next section) is based on the method of distinguishing statistics. More precisely, we use the following fact, whose proof can be found for instance in~\cite[Proposition 7.12]{levin2017markov}.

\begin{lemma}\label{lem:distinguishing statistics}
    Let $P$ and $Q$ be two probability distributions over some finite space $\cX$ and let $f\colon \cX\to\R$ be some observable. We have
    \[
    d_\tv(P,Q)\geq \frac{\left(\E_Q f -\E_P f\right)^2}{\left(\E_Q f -\E_P f\right)^2+2\Var_P(f)+2\Var_Q(f)}\, \cdot 
    \]
\end{lemma}

In the rooted case, our distinguishing observable is based on the size of the two subtrees separated by the root-edge $\rho$. Let $(\rho_+,\rho_-)$ be an arbitrary orientation of $\rho$, and let $T_n^+$ (resp. $T_n^-$) the tree containing $\rho_+$ (resp. $\rho_-$) when edge $\rho$ is removed from $T_n$. A natural question is to determine whether the sequence $(|T_n^+|/n)_{n\geq 1}$ converges, and, if so, whether the limiting distribution depends on $q$. The fact that it converges can be seen by noticing that $(|T_n^+|/n)_{n\geq 1}$ is a martingale. Indeed, for $\Lambda\in \{A,B\}$, let $|T_{n,\Lambda}^+|$ (resp. $|T_{n,\Lambda}^-|$) be the number of nodes of type $\Lambda$ in $T_n^+$ (resp. in $T_n^-$), and let $(\cF_n)_{n\geq 1}$ denotes the filtration keeping track of the whole history (including types and time-labels). Then, for $\Lambda\in \{A,B\}$, we have
\[
\E\left[|T_{n+1,\Lambda}^+| -|T_{n,\Lambda}^+| \,|\, \cF_n\right]=(1-q)\frac{|T_{n,\Lambda}^+|}{n}+q \frac{|T_{n,\bar{\Lambda}}^+|}{n}\, ,
\]
where $\bar{\Lambda}$ is such that $\{\Lambda,\bar{\Lambda}\}=\{A,B\}$. Taking the sum over $\Lambda\in\{A,B\}$ in this equality, we obtain
\[
\E\left[|T_{n+1}^+| -|T_{n}^+| \,|\, \cF_n\right]=\frac{|T_{n}^+|}{n}\, .
\]
Hence $(|T_n^+|/n)_{n\geq 1}$ is a martingale, and since it is positive, it converges almost surely. While we were not able to determine the limiting distribution, Proposition~\ref{prop:diff-expectation-rooted} below entails that it depends on $q$ in a non-trivial way. By symmetry, we have $\E_q[|T_n^+|]=n$ for all $q\in [0,1]$, but the dependence on $q$ will appear in the second moment of $|T_n^+|$. In order to lay the groundwork for the next section, we will rather consider $\E_q[|T_n^+||T_n^-|]$ instead of $\E_q[|T_n^+|^2]$, which amounts to the same thing since
\[
\E_q[|T_n^+||T_n^-|]=\E_q[|T_n^+|(2n-|T_n^+|)]=2n^2-\E_q[|T_n^+|^2]\, .
\]

\begin{proposition}\label{prop:diff-expectation-rooted}
    For all $0\leq q_0<q_1\leq 1/2$ and for all $n\geq 2$, we have
\[
\E_{q_0}[|T_n^+||T_n^-|]-\E_{q_1}[|T_n^+||T_n^-|]\geq \delta(q_0,q_1)n^2\, ,
\]
where
\[
\delta(q_0,q_1)=\frac{1}{3}(q_1-q_0)(1-q_1-q_0)>0\, .
\]
\end{proposition}

Before proving Proposition~\ref{prop:diff-expectation-rooted}, let us see how it implies Theorem~\ref{thm:rooted-unlabelled-BCMRT}.

\begin{proof}[Proof of Theorem~\ref{thm:rooted-unlabelled-BCMRT}]
We use Lemma~\ref{lem:distinguishing statistics} with $f(T_n)=|T_n^+||T_n^-|$. By Proposition~\ref{prop:diff-expectation-rooted}, we have 
\begin{equation*}
\E_{q_0}[f] -\E_{q_1}[f]\geq \delta(q_0,q_1)n^2\, .\label{eq:11}
\end{equation*}
Moreover, since
\[
|T_n^+| |T_n^-|=4n^2\frac{T_n^+}{2n}\left(1-\frac{T_n^+}{2n}\right)\leq n^2\, ,
\]
we have, for all $q\in [0,1/2]$,
\begin{equation*}
\Var_q(f)\leq \E_q[f^2]\leq n^4\, .\label{eq:12}
\end{equation*}
Hence the total-variation distance is larger than $\frac{\delta(q_0,q_1)^2}{\delta(q_0,q_1)^2+4}$, and the optimal risk is bounded away from $1/2$.
\end{proof}

Let us now investigate the dynamic of $|T_n^+| |T_n^-|$ so as to prove Proposition~\ref{prop:diff-expectation-rooted}. Recall that $|T_{n,\Lambda}^+|$ (resp. $|T_{n,\Lambda}^-|$) is the number of nodes of type $\Lambda\in\{A,B\}$ in $T_n^+$ (resp. in $T_n^-$) and that $(\cF_n)_{n\geq 1}$ denotes the filtration keeping track of the whole history. For $\{\Lambda,\bar{\Lambda}\}=\{A,B\}$, conditionally on $\cF_n$, the probability that node $(n+1,\Lambda)$ attaches to $T_n^+$ is
\begin{equation}\label{eq:def-l_A}
w_{n,\Lambda}^+=(1-q)\frac{T_{n,\Lambda}^+}{n}+q\frac{T_{n,\bar{\Lambda}}^+}{n}\, ,
\end{equation}
and the probability that it attaches to $T_n^-$ is $w_{n,\Lambda}^-=1-w_{n,\Lambda}^+$. Hence, conditionally on $\cF_n$, we have
\[
|T_{n+1}^+||T_{n+1}^-| -|T_{n}^+| |T_{n}^-|=
\begin{cases}
2|T_n^-| &\text{with probability $w_{n,A}^+w_{n,B}^+$,}\\
2|T_n^+| &\text{with probability $w_{n,A}^-w_{n,B}^-$,}\\
|T_n^+|+|T_n^-|+1 &\text{with probability $w_{n,A}^+w_{n,B}^- +w_{n,A}^-w_{n,B}^+$\, ,}
\end{cases}
\]
from which we get
\begin{align*}
\E_q\left[ |T_{n+1}^+||T_{n+1}^-|\,|\, \cF_n\right] &=\left(1+\frac{2}{n}\right)|T_{n}^+||T_{n}^-| +\sum_{\Lambda\in\{A,B\}}w_{n,\Lambda}^+w_{n,\bar{\Lambda}}^-\, ,
\end{align*}
where we used that $w_{n,A}^++w_{n,B}^+=\frac{|T_n^+|}{n}$ and $w_{n,A}^-+w_{n,B}^-=\frac{|T_n^-|}{n}$. Moreover, defining
\[
R_n=\sum_{\Lambda\in\{A,B\}} |T_{n,\Lambda}^+||T_{n,\bar{\Lambda}}^-|\, ,
\]
we have
\begin{align}
\sum_{\Lambda\in\{A,B\}}w_{n,\Lambda}^+w_{n,\bar{\Lambda}}^-
&= \frac{(1-q)^2+q^2}{n^2}R_n+\frac{2q(1-q)}{n^2}(|T_{n}^+||T_{n}^-|-R_n)\nonumber\\
&= \frac{2q(1-q)}{n^2}|T_{n}^+||T_{n}^-|+\frac{(1-2q)^2}{n^2}R_n\label{eq:prod-w}\, .
\end{align}

Hence, taking expectations, we obtain
\[
\E_q\left[ |T_{n+1}^+||T_{n+1}^-|\right] =\left(1+\frac{2}{n}+\frac{2q(1-q)}{n^2}\right)\E_q\left[|T_{n}^+||T_{n}^-|\right] +\frac{(1-2q)^2}{n^2}\E_q\left[R_n\right]\, .
\]
We now look for a recurrence relation for the sequence $\E_q[R_n]$. Conditionally on $\cF_n$, we have
\[
R_{n+1}-R_n=
\begin{cases}
|T_n^-| &\text{with probability $w_{n,A}^+w_{n,B}^+$}\\
|T_n^+| &\text{with probability $w_{n,A}^-w_{n,B}^-$}\\
|T_{n,A}^+|+|T_{n,B}^-|+1 &\text{with probability $w_{n,A}^+w_{n,B}^-$}\\
|T_{n,A}^-|+|T_{n,B}^+|+1 &\text{with probability $w_{n,A}^-w_{n,B}^+$}\\
\end{cases}
\]
Rearranging and invoking~\eqref{eq:prod-w}, we obtain
\begin{align*}
\E_q\left[ R_{n+1}\,|\, \cF_n\right] &= \left(1+\frac{2(1-2q)}{n}+\frac{(1-2q)^2}{n^2}\right)R_n+\left(\frac{2q}{n}+\frac{2q(1-q)}{n^2}\right)|T_{n}^+||T_{n}^-|\, .
\end{align*}
Denoting $f_q(n)=\E_q[|T_{n}^+||T_{n}^-|]$ and $g_q(n)=\E_q[R_n]$, we obtain the following recurrence system:
\begin{equation}\label{eq:rec-system}
\begin{cases}
&f_q(n+1)=a_n(q)f_q(n)+b_n(q)g_q(n)\\
&g_q(n+1)=c_n(q)g_q(n)+d_n(q)f_q(n)\\
&f_q(1)=1\, ,\quad g_q(1)=1\, .
\end{cases}
\end{equation}
where
\begin{equation}\label{eq:coeff}
\begin{aligned}
      a_n(q) &= 1+\frac{2}{n}+\frac{2q(1-q)}{n^2}\\
    b_n(q) &=\frac{(1-2q)^2}{n^2}\\
    c_n(q) &= 1+\frac{2(1-2q)}{n}+\frac{(1-2q)^2}{n^2}\\
    d_n(q) &= \frac{2q}{n}+\frac{2q(1-q)}{n^2} \, .
    \end{aligned}
\end{equation}

Let us state one useful fact that will be used in the proof of Proposition~\ref{prop:diff-expectation-rooted}.

\begin{lemma}\label{lem:link_f-g}
    For all $n\geq 1$, we have $|T_n^+||T_n^-|\leq 2R_n$. In particular, $f_q(n)\leq 2g_q(n)$.
\end{lemma}
\begin{proof}[Proof of Lemma~\ref{lem:link_f-g}]
  We have
  \begin{align*}
      R_n&=|T_{n,A}^+|(n-|T_{n,B}^+|)+|T_{n,B}^+|(n-|T_{n,A}^+|)\\
      &=n|T_n^+|-2|T_{n,A}^+||T_{n,B}^+|\, .
      \end{align*}
      Using that $u(1-u)\leq 1/4$ for $u\in [0,1]$, we have $|T_{n,A}^+||T_{n,B}^+|\leq |T_n^+|^2/4$. Hence,
      \[
R_n\geq n|T_n^+|-\frac{|T_n^+|^2}{2}=\frac{|T_n^+|}{2}\left(2n-|T_n^+|\right)=\frac{|T_n^+||T_n^-|}{2}\, \cdot  
      \]
\end{proof}

We are now ready to prove Proposition~\ref{prop:diff-expectation-rooted}.

\begin{proof}[Proof of Proposition~\ref{prop:diff-expectation-rooted}]
Let $(f_q(n))_{n\geq 1}$ and $(g_q(n))_{n\geq 1}$ be defined by~\eqref{eq:rec-system} and let $0\leq q_0<q_1\leq 1/2$. We will show by induction that for all $n\geq 2$,
\begin{equation}\label{eq:rec-to-show}
\begin{cases}
   & f_{q_0}(n)-f_{q_1}(n)\geq (f_{q_0}(2)-f_{q_1}(2))\prod_{i=2}^{n-1}a_i(q_0)\\
   & g_{q_0}(n)-g_{q_1}(n)\geq 0\, ,
\end{cases}
\end{equation}
The property is clearly true for $n=2$, with the convention that $\prod_{i=2}^{i=1}a_i(q_0)=1$. Assume the property holds for some $n\geq 2$. We have
\begin{align*}
    f_{q_0}(n+1)-f_{q_1}(n+1)&= a_n(q_0)\left(f_{q_0}(n)-f_{q_1}(n)\right)+f_{q_1}(n)\left(a_n(q_0)-a_n(q_1)\right)\\
    &\quad  +b_n(q_0)\left(g_{q_0}(n)-g_{q_1}(n)\right)+g_{q_1}(n)\left(b_n(q_0)-b_n(q_1)\right)\, .
\end{align*}
Using the second part of the induction hypothesis and the fact that $b_n(q_0)\geq 0$, we have
\[
b_n(q_0)\left(g_{q_0}(n)-g_{q_1}(n)\right)\geq 0\, .
\]
Next, observe that
\[
a_n(q_1)-a_n(q_0)=\frac{2}{n^2}\left(q_1(1-q_1)-q_0(1-q_0)\right)=\frac{b_n(q_0)-b_n(q_1)}{2}\geq 0\, .
\]
Hence,
\[
f_{q_0}(n+1)-f_{q_1}(n+1)\geq a_n(q_0)\left(f_{q_0}(n)-f_{q_1}(n)\right)+\left(b_n(q_0)-b_n(q_1)\right)\left(g_{q_1}(n)-\frac{f_{q_1}(n)}{2}\right)\, .
\]
By Lemma~\ref{lem:link_f-g}, we have $2g_{q_1}(n)\geq f_{q_1}(n)$, which yields
\[
    f_{q_0}(n+1)-f_{q_1}(n+1)\geq  a_n(q_0)\left(f_{q_0}(n)-f_{q_1}(n)\right)\, .
\]
Let us now show that $g_{q_0}(n+1)-g_{q_1}(n+1)\geq 0$. By the induction hypothesis, we have
\begin{align*}
    g_{q_0}(n+1)&\geq  c_n(q_0)g_{q_1}(n)+d_n(q_0)f_{q_1}(n)\\
    &= g_{q_1}(n+1)+ \left(c_n(q_0)-c_n(q_1)\right)g_{q_1}(n)+
    \left(d_n(q_0)-d_n(q_1)\right)f_{q_1}(n)\, .
\end{align*}
Observing that
\[
d_n(q_1)-d_n(q_0)=\frac{2}{n}(q_1-q_0)+\frac{2}{n^2}\left(q_1(1-q_1)-q_0(1-q_0)\right)=\frac{c_n(q_0)-c_n(q_1)}{2}\geq 0\, ,
\]
we have
\[
g_{q_0}(n+1)\geq g_{q_1}(n+1)+ \left(c_n(q_0)-c_n(q_1)\right)\left(g_{q_1}(n)-\frac{f_{q_1}(n)}{2}\right)\, .
\]
Invoking Lemma~\ref{lem:link_f-g} again, this concludes the proof of~\eqref{eq:rec-to-show}.
Now, we have
\[
\prod_{i=2}^{n-1}a_i(q_0)=\prod_{i=2}^{n-1}\frac{i^2+2i+2q(1-q)}{i^2}\geq \prod_{i=2}^{n-1}\frac{i+2}{i}=\frac{n(n+1)}{6}\geq \frac{n^2}{6}\, \cdot 
\]
The proof can now be concluded by noticing that $f_{q_0}(2)-f_{q_1}(2)=2(q_1-q_0)(1-q_1-q_0)$.
\end{proof}

\section{Testing in unrooted unlabelled BCMRT}\label{sec:unrooted}

We now observe a tree $T_n$ of size $2n$ without types, without time labels, and without the root edge. We show that it is still possible to distinguish between two distinct parameters $q_0$ and $q_1$, with a risk strictly less than $1/2$.

\begin{theorem}\label{thm:unrooted-unlabelled-BCMRT}
Assume that one observes an unrooted unlabelled {\rm BCMRT}. Then, for all $0\leq q_0<q_1\leq 1/2$, we have
\[
\sup_{n\geq 2} \bR^\star_{q_0,q_1} <\frac{1}{2}\, .
\]
\end{theorem}

Recall that in the rooted setting, our distinguishing statistic was $f(T_n)=|T_n^+||T_n^-|$, the product of sizes of the two subtrees around the root edge. Since we can not locate the root edge anymore, we will instead take the sum of such products over all the edges of the tree. Namely, for each edge $e\in E_n$, we take an arbitrary orientation $e=(e_-,e_+)$ and denote by $T_n^{e_+}$ (resp. $T_n^{e_-}$) the tree that contains $e_+$ (resp. $e_-$) after removing the edge $e$. We then define the variable
\[
S_n=\sum_{e\in E_n} |T_n^{e_+}||T_n^{e_-}|\, .
\]
Note that $S_n$ is equal to the sum of all pairwise distances in the tree, since each edge is counted as many times as the pairs of nodes it separates. 

Let us look for a recurrence relation for $\E_q[S_n]$. We have
\[
\E_q\left[S_{n+1}\,|\, \cF_n\right]=2(2n+1)+\sum_{e\in E_n}\E_q\left[ |T_{n+1}^{e_+}|\cdot |T_{n+1}^{e_-}|\,|\, \cF_n\right]\, .
\]
Then, letting $w_{n, \Lambda}^{e_+}$ be the conditional probability, given $\cF_n$, that node $(n+1, \Lambda)$, for $ \Lambda=A,B$, attaches in $T_n^{e_+}$, and $w_{n, \Lambda}^{e_-}$ the probability that it attaches in $T_n^{e_-}$, we have that that
$\E_q\left[ |T_{n+1}^{e_+}|\cdot |T_{n+1}^{e_-}|\,|\, \cF_n\right]$ is equal to
\begin{align*}
 w_{nA}^{e+}w^{e+}_{nB}(|T_n^{e+}|+2)|T_{n}^{e-}|  +  w_{nA}^{e-}w^{e-}_{nB}(|T_n^{e-}|+2)|T_{n}^{e+}|+(w_{nA}^{e+}w^{e-}_{nB}+w_{nA}^{e-}w^{e+}_{nB})(|T_n^{e+}|+1)(|T_{n}^{e-}|+1)
\end{align*}
Therefore, denoting by $\overline{\Lambda}$ the complement of $\Lambda$ in $\{A,B\}$, we have that
\begin{align}
   \E_q\left[S_{n+1}\,|\, \cF_n\right]   &= 2(2n+1)+\left(1+\frac{2}{n}\right)S_n +\sum_{e\in E_n}\sum_{ \Lambda\in\{A,B\}} w_{n, \Lambda}^{e_+}w_{n,\bar{ \Lambda}}^{e_-}\label{eq:recSn-first}\, ,
\end{align}
where the overline on $ \Lambda$ denotes its complement
in the set $\{A, B\}$.
Letting $|T_{n, \Lambda}^{e_+}|$ (resp. $|T_{n, \Lambda}^{e_-}|$) be the number of nodes of type $ \Lambda$ in $T_n^{e_+}$ (resp. in $T_n^{e_-}$), we have
\begin{align*}
    \sum_{ \Lambda\in\{A,B\}} w_{n, \Lambda}^{e_+}w_{n,\bar{ \Lambda}}^{e_-} &= \frac{ (1-q)^2+q^2}{n^2}\sum_{ \Lambda\in\{A,B\}}|T_{n, \Lambda}^{e_+}||T_{n,\bar{ \Lambda}}^{e_-}|+\frac{2q(1-q)}{n^2}\sum_{ \Lambda\in\{A,B\}}|T_{n, \Lambda}^{e_+}||T_{n, \Lambda}^{e_-}|\, .
\end{align*}
Similarly to~\eqref{eq:prod-w}, we obtain
\begin{equation}
\sum_{e\in E_n}\sum_{ \Lambda\in\{A,B\}} w_{n, \Lambda}^{e_+}w_{n,\bar{ \Lambda}}^{e_-} =\frac{2q(1-q)}{n^2}S_n +\frac{(1-2q)^2}{n^2}K_n\, ,\label{former-w}
\end{equation}
with
\[
K_n=\sum_{e\in E_n}\sum_{ \Lambda\in\{A,B\}}|T_{n, \Lambda}^{e_+}||T_{n,\bar{ \Lambda}}^{e_-}|\, .
\]
Coming back to~\eqref{eq:recSn-first}, we get
\[
\E_q\left[S_{n+1}\,|\,\cF_n\right]=2(2n+1)+\left(1+\frac{2}{n}+\frac{2q(1-q)}{n^2}\right)S_n+\frac{(1-2q)^2}{n^2}K_n\, .
\]
Next, we look for a recurrence relation for $\E_q[K_n]$. Working in a similar way as for $S_n$, we recover that
\[
\E_q \left[K_{n+1}\,|\,\mathcal{F}_n\right]\,=\,  2(n+1)+K_n+\sum _{e\in E_n}\sum _{ \Lambda\in\{A,B\}}\big\{|T^{e_+}_{n, \Lambda}| w^{e_-}_{n,\bar{ \Lambda}}+|T^{e_-}_{n,\bar{ \Lambda}}| w^{e_+}_{n, \Lambda}\big\}+\sum _{e\in E_n}\sum _{ \Lambda\in\{A,B\}}w^{e_+}_{n, \Lambda}w^{e_-}_{n,\bar{ \Lambda}}\, .\label{eq:103}
\]
Writing $w_{n,\Lambda}^{e-}=\frac{(1-q)|T^{e_-}_{n, \Lambda}|}{n}+\frac{q|T^{e_-}_{n, \overline{\Lambda}}|}{n}$, and similarly for $w_{n,\Lambda}^{e+}$, we obtain that
\[
\sum _{e\in E_n}\sum _{ \Lambda\in\{A,B\}}\big\{|T^{e_+}_{n, \Lambda}| w^{e_-}_{n,\bar{ \Lambda}}+|T^{e_-}_{n,\bar{ \Lambda}}| w^{e_+}_{n, \Lambda}\big\}=\frac{2(1-2q)}{n}K_n+\frac{2q}{n}S_n\, ,
\]
Combining the last two displays with~\eqref{former-w}, we arrive at
\[
\E_q\left[K_{n+1}\,|\, \cF_n\right]=2(n+1)+\left(1+\frac{2(1-2q)}{n}+\frac{(1-2q)^2}{n^2}\right)K_n +\left(\frac{2q}{n}+\frac{2q(1-q)}{n^2}\right)S_n\, .
\]
Taking expectations, we obtain the following recurrence system:
\begin{equation}\label{eq:rec-system-unrooted}
\begin{cases}
&\E_q[S_{n+1}]=a_n(q)\E_q[S_n]+b_n(q)\E_q[K_n]+2(2n+1)\\
&\E_q[K_{n+1}]=c_n(q)\E_q[K_n]+d_n(q)\E_q[S_n]+2(n+1)\\
&f_q(1)=1\, ,\quad g_q(1)=1\, ,
\end{cases}
\end{equation}
where $a_n(q), b_n(q), c_n(q), d_n(q)$ are the same as in~\eqref{eq:coeff}.

We now establish the following asymptotic equivalence.

\begin{proposition}\label{prop:asympt-unrooted}
For all $q\in [0,1/2]$, we have
\[
\E_q[S_n]\underset{n\to \infty}{\sim} 4n^2\log n\, .
\]
\end{proposition}

Before proving Proposition~\ref{prop:asympt-unrooted}, we first state a useful fact that will serve as an analogue of Lemma~\ref{lem:link_f-g} of the paper. 

\begin{lemma}\label{lem:link_f-g_unrooted}
    For all $n\geq 1$, we have $S_n\leq 2K_n$. In particular, $\E_q[S_n]\leq 2\E_q[K_n]$.
\end{lemma}
\begin{proof}[Proof of Lemma~\ref{lem:link_f-g_unrooted}]
    The proof of Lemma~\ref{lem:link_f-g} can be extended to any edge $e$:
    \[
    |T_{n,A}^{e_+}||T_{n,B}^{e_-}|+|T_{n,B}^{e_+}||T_{n,A}^{e_-}|\geq \frac{|T_n^{e_+}||T_n^{e_-}|}{2}\, ,
    \]
    and it suffices to sum over all edges $e\in E_n$ to get the desired bound. 
\end{proof}

We are now ready to prove Proposition~\ref{prop:asympt-unrooted}.

\begin{proof}[Proof of Proposition~\ref{prop:asympt-unrooted}]
Expanding the relation for $\E_q[S_n]$ in~\eqref{eq:rec-system-unrooted}, we have
    \begin{align*}
\E_q[S_{n+1}]&=\underbrace{\prod_{i=1}^n a_i(q)}_{(\mathbf{I})} +\underbrace{\sum_{i=1}^n \left(\prod_{j=i+1}^n a_j(q)\right)b_i(q)\E_q[K_i]}_{(\mathbf{II})}+\underbrace{2\sum_{i=1}^n \left(\prod_{j=i+1}^n a_j(q)\right)(2i+1)}_{(\mathbf{III})}\, .
     \end{align*}
    First, we see that 
    \[
    (\mathbf{I})=\prod_{i=1}^n\frac{i^2+2i+2q(1-q)}{i^2}\leq \prod_{i=1}^n\frac{(i+1)^2}{i^2}=(n+1)^2\, .
    \]
    To bound $(\mathbf{II})$, let us first show that for all $q\in [0,1/2]$ and for all $n\geq 1$, we have
    \[
    \E_q[K_n]\leq n^2+2n^2\log n\, .
    \]
    This holds for $n=1$ and, by Lemma~\ref{lem:link_f-g_unrooted}, we have, for all $n\geq 1$,
    \begin{align*}
        \E_q[K_{n+1}]&=\left(1+\frac{2(1-2q)}{n}+\frac{(1-2q)^2}{n^2}\right)\E_q[K_n]+\left(\frac{2q}{n}+\frac{2q(1-q)}{n^2}\right)\E_q[S_n]+2(n+1)\\
        &\leq \left(1+\frac{2}{n}+\frac{1}{n^2}\right)\E_q[K_n]+2(n+1)\, .
    \end{align*}
    Iterating this inequality, we get
    \begin{align*}
        \E_q[K_{n+1}]&\leq \prod_{i=1}^n \left(1+\frac{2}{i}+\frac{1}{i^2}\right)+2\sum_{i=1}^n(i+1)\prod_{j=i+1}^n  \left(1+\frac{2}{j}+\frac{1}{j^2}\right)\\
        &= \prod_{i=1}^n\frac{(i+1)^2}{i^2}+2\sum_{i=1}^n \frac{(n+1)^2}{i+1}\,\leq\, (n+1)^2+2(n+1)^2\log(n+1)\, ,
    \end{align*}
    where for the last inequality, we used that $\sum_{i=2}^{n+1}\frac{1}{i}\leq\log(n+1)$. Hence, since $b_i(q)\leq 1/i^2$, we have
    \begin{align*}
        (\mathbf{II})&\leq \sum_{i=1}^n \left(\prod_{j=i+1}^n a_j(q)\right)\left(1+2\log(i)\right) \, .
    \end{align*}
Using that 
\[
\prod_{j=i+1}^n a_j(q)\leq \prod_{j=i+1}^n \left(1+\frac{2}{j}+\frac{1}{j^2}\right)=\frac{(n+1)^2}{(i+1)^2}\, ,
\]
we then have
\begin{align*}
      (\mathbf{II})&\leq (n+1)^2\sum_{i=1}^n \frac{1+2\log(i)}{(i+1)^2}\, =\, O(n^2)\, .
    \end{align*}
Moving to $(\mathbf{III})$, we first observe that, for $\gamma=\sqrt{1-2q(1-q)}$,
\begin{align*}
\prod_{i=1}^n a_i(q) &\,=\, \frac{1}{\Gamma(n+1)^2}\prod_{i=1}^{n}\left(i+1+\gamma\right)\left(i+1-\gamma\right)\\
&\,=\,\frac{\Gamma(n+2+\gamma)\Gamma(n+2-\gamma)}{\Gamma(2+\gamma)\Gamma(2-\gamma)\Gamma(n+1)^2}\, ,
\end{align*}
where $\Gamma$ is the Gamma function. Since $\Gamma(n+2+\gamma)\Gamma(n+2-\gamma)\underset{n\to \infty}{\sim}\Gamma(n+2)^2$ (which can be checked using Stirling's formula), we get
\[
        (\mathbf{III})\,=\, 2\sum_{i=1}^n \frac{\Gamma(n+2+\gamma)\Gamma(n+2-\gamma)}{\Gamma(n+1)^2}\cdot \frac{\Gamma(i+1)^2(2i+1)}{\Gamma(i+2+\gamma)\Gamma(i+2-\gamma)}\,\underset{n\to\infty}{\sim}\, 2n^2\sum_{i=1}^n\varphi(i)\, ,
\]
with $\varphi(i)=\frac{\Gamma(i+1)^2(2i+1)}{\Gamma(i+2+\gamma)\Gamma(i+2-\gamma)}$. One way to obtain the asymptotics for $\sum_{i=1}^n \varphi(i)$ is to first consider the first $\lceil \log n\rceil$ terms and use that $\log\Gamma$ is convex so that
    \[
\sum_{i=1}^{\lceil\log n\rceil}\varphi(i)\leq \sum_{i=1}^{\lceil\log n\rceil}\frac{2i+1}{(i+1)^2}=O(\log\log n)\, .
    \]
    For the remaining terms, we have
 \[
 \left|\sum_{i=\lceil\log n\rceil +1}^n\left(\varphi(i)- \frac{2i+1}{(i+1)^2}\right)\right|\leq \sup_{k> \lceil\log n\rceil }\left| \frac{\varphi(k)(k+1)^2}{2k+1}-1\right|\sum_{i=\lceil\log n\rceil +1}^n \frac{2i+1}{(i+1)^2}\,=\, o(\log n)\, ,
\]
    since $\sup_{k> \lceil\log n\rceil }\left| \frac{\varphi(k)(k+1)^2}{2k+1}-1\right|=o(1)$. Hence,
    \[
(\mathbf{III})\underset{n\to\infty}{\sim}2n^2\sum_{i=\lceil\log n\rceil +1}^n \frac{2i+1}{(i+1)^2}\underset{n\to\infty}{\sim}4n^2\log n\, ,
    \]
    which concludes the proof.
\end{proof}

As far as inference is concerned, Proposition~\ref{prop:asympt-unrooted} may be seen as a negative result: $\E_q[S_n]$ does not depend on $q$ at first order. This indicates that, in order to apply Lemma~\ref{lem:distinguishing statistics}, we do need more subtle bounds on the variance than in the rooted case. Indeed, in that case, a crude bound on the second moment of $|T_n^+||T_n^-|$ was sufficient, while Proposition~\ref{prop:asympt-unrooted} implies that the second moment of $S_n$ is at least of order $n^4\log^2(n)$, which is much larger than the squared difference between expectations. Theorem~\ref{thm:unrooted-unlabelled-BCMRT} will be a consequence of the two following propositions.

\begin{proposition}\label{prop:diff-unrooted}
For all $0\leq q_0<q_1\leq 1/2$ and for all $n\geq 2$, we have
\[
\E_{q_0}[S_n]-\E_{q_1}[S_n]\geq \delta(q_0,q_1)n^2\, ,
\]
where $\delta(q_0,q_1)>0$ is the same as in Proposition~\ref{prop:diff-expectation-rooted}.
\end{proposition}
\begin{proof}[Proof of Theorem~\ref{prop:diff-unrooted}]
Even though the recurrence systems~\eqref{eq:rec-system}  and~\eqref{eq:rec-system-unrooted}  are different because of the additional linear terms, the coefficients $a_n(q),b_n(q),c_n(q),d_n(q)$ are the same in both systems and the proof of Proposition~\ref{prop:diff-expectation-rooted} can be emulated to prove Proposition~\ref{prop:diff-unrooted}, invoking Lemma~\ref{lem:link_f-g_unrooted} instead of Lemma~\ref{lem:link_f-g}.
\end{proof}

\begin{proposition}\label{prop:variance:unrooted}
    For all $q\in [0,1/2]$, we have
    \[
    \Var_q(S_n)=O(n^4)\, .
    \]
\end{proposition}

\begin{proof}[Proof of Theorem~\ref{prop:variance:unrooted}]
  Note that the tree $T_n$ can be written as a function of $2(n-1)$ independent random variables:
  \begin{equation}\label{eq:tree-function}
  T_n=f\left((b_{2,A},U_{2,A}), (b_{2,B},U_{2,B}),\dots, (b_{n,A},U_{n,A}), (b_{n,B},U_{n,B})\right)\, ,
  \end{equation}
  where for $i=2,\dots,n$ and $\Lambda=A,B$, 
  \[
(b_{i,\Lambda},U_{i,\Lambda})\sim \text{Bernoulli}(q)\otimes \text{Unif}\left(\{1,\dots,i-1\}\right)\, .
  \]
  If $b_{i,\Lambda}=0$, then node $(i,\Lambda)$ chooses its parent in community $\Lambda$. If $b_{i,\Lambda}=1$, then it chooses its parent in the other community. The variable $U_{i,\Lambda}$ then gives the time label of the parent. Let $T_n'(i,\Lambda)$ be defined as in~\eqref{eq:tree-function}, except that $(b_{i,\Lambda},U_{i,\Lambda})$ is replaced by an independent copy $(b'_{i,\Lambda},U'_{i,\Lambda})$. Let $p(i,\Lambda)$ be the parent of $(i,\Lambda)$ in $T_n$, and $p'(i,\Lambda)$ its new parent in $T'_n(i,\Lambda)$ (note that all the other parents remain unchanged). 
  Let us now investigate how this change affects the variable $S_n$. Note that when edge $e$ does not belong to the path $\cP_{i,\Lambda}$ connecting $p(i,\Lambda)$ and $p'(i,\Lambda)$, then the value of $|T_n^{e_+}| |T_n^{e_-}|$ is unchanged. And when $e$ belongs to $\cP_{i,\Lambda}$, then, assuming without loss of generality that $p(i,\Lambda)$ belongs to $T_n^{e_+}$, we have that $|T_n^{e_+}||T_n^{e_-}|$ becomes 
  \[
  \left(|T_n^{e_+}|-|T^+_n(i,\Lambda)|\right)\left(|T_n^{e_-}|+|T^+_n(i,\Lambda)|\right)\, ,
  \]
   where $T^+_n(i,\Lambda)$ is the subtree rooted at node $(i,\Lambda)$ and containing its descendants up to generation $n$. Hence, by the Efron--Stein inequality~\cite{efron1981jackknife, steele1986efron}, we have
  \begin{align*}
      \Var_q(S_n)&\leq \frac{1}{2}\sum_{\Lambda\in\{A,B\}}\sum_{i=2}^n \E_q\left[\left(S_n-S'_n(i,\Lambda)\right)^2\right]\\
      &=\frac{1}{2}\sum_{\Lambda\in\{A,B\}}\sum_{i=2}^n \E_q\left[\left(\sum_{e\in\cP_{i,\Lambda}}|T^+_n(i,\Lambda)|\left(|T_n^{e_-}|-|T_n^{e_+}|+|T^+_n(i,\Lambda)|\right) \right)^2\right]\, .
  \end{align*}
  Using that
  \[
\left||T_n^{e_-}|-|T_n^{e_+}|+|T^+_n(i,\Lambda)| \right|\leq |T_n^{e_-}|+|T_n^{e_+}|-|T^+_n(i,\Lambda)|\leq 2n\, ,
  \]
  we get
   \begin{align*}
      \Var_q(S_n)&\leq 2n^2\sum_{\Lambda\in\{A,B\}}\sum_{i=2}^n \E_q\left[|\cP_{i,\Lambda}|^2|T^+_n(i,\Lambda)|^2 \right]\, ,
  \end{align*}
  where $|\cP_{i,\Lambda}|$ is the length of $\cP_{i,\Lambda}$, i.e., the distance between $p(i,\Lambda)$ and $p'(i,\Lambda)$. Since, $|T^+_n(i,\Lambda)|$ is independent from the whole history up to time $i$ (including the knowledge of $p(i,\Lambda)$ and $p'(i,\Lambda)$), the variables $|T^+_n(i,\Lambda)|$ and $|\cP_{i,\Lambda}|$ are independent and we have
  \[
  \E_q\left[|\cP_{i,\Lambda}|^2|T^+_n(i,\Lambda)|^2 \right]=\E_q\left[|\cP_{i,\Lambda}|^2\right]\E_q\left[|T^+_n(i,\Lambda)|^2 \right]\, .
  \]
Note that
  \[
  |\cP_{i,\Lambda}|\leq (L_{i,\Lambda}-1)+(L'_{i,\Lambda}-1)+1\leq L_{i,\Lambda}+L'_{i,\Lambda}\, ,
  \]
  where $L_{i,\Lambda}$ (resp. $L'_{i,\Lambda}$) is the level of $(i,\Lambda)$ in $T_n$ (resp. in $T_n'(i,\Lambda)$). The proof can now be concluded using two lemmas that will be proved subsequently. By Lemma~\ref{lem:distance-to-root},
  \[
  \E_q\left[|\cP_{i,\Lambda}|^2\right]\leq 4\E_q[L_{i,\Lambda}^2]\leq 8H_{i-1}^2\leq 8\left(\log(i)+1\right)^2\, .
  \]
 Moreover, by Lemma~\ref{lem:size-Ti}, we have
 \[
 \E_q\left[|T^+_n(i,\Lambda)|^2\right]\leq \frac{2n^2}{i^2}\, \cdot 
 \]
 Hence we obtain
  \[
  \Var_q(S_n)\leq 64 n^4\sum_{i=2}^n \frac{\left(\log(i)+1\right)^2}{i^2}=O(n^4)\, .
  \]
\end{proof}
\begin{lemma}\label{lem:distance-to-root}
    For $\Lambda\in\{A,B\}$ and $n\geq 2$, let $L_{n,\Lambda}$ be the level of $(n,\Lambda)$ in $T_n$, i.e., its distance to the closest root, $(1,\Lambda)$ or $(1,\bar{\Lambda})$. Then, for all $q\in [0,1/2]$, we have
    \[
    \E_q[L_{n,\Lambda}^2]=H_{n-1}^2+H_{n-1}-\sum_{j=1}^{n-1}\frac{1}{j^2}\, ,
    \]
    where $H_{n-1}=\sum_{j=1}^{n-1}\frac{1}{j}$.
\end{lemma}
\begin{proof}[Proof of Theorem~\ref{lem:distance-to-root}]
Whatever $ \Lambda$, the parent of $(n, \Lambda)$ can be written $(u_1, \Lambda_1)$, where $u_1$ is some uniformly chosen number between $1$ and $i-1$. Then, whatever $ \Lambda_1$, if $u_1>1$, the parent of $(u_1, \Lambda_1)$ can be written $(u_2, \Lambda_2)$, where $u_2$ is some uniformly chosen number between $1$ and $u_1-1$. If we continue until we find $u_m=1$, then we found the closest root, and the level of $(n, \Lambda)$ is equal to $m$. As observed by Devroye~\cite{devroye1988applications}, this is exactly the distribution of the number of records in a uniform random permutation of $\{1,\dots,n-1\}$. In particular,
\[
\E_q[L_{n, \Lambda}]=H_{n-1}\quad \text{ and }\quad \Var_q(L_{n, \Lambda})=H_{n-1}-\sum_{j=1}^{n-1}\frac{1}{j^2}\, .
\]
\end{proof}

\begin{lemma}\label{lem:size-Ti}
    For $1\leq i\leq n$ and $\Lambda\in\{A,B\}$, let $T_n^+(i)$ be the subtree rooted at $(i,\Lambda)$ including all the descendants of $(i,\Lambda)$ up to time $n$. Then, for all $q\in [0,1/2]$, we have
    \[
    \E_q\left[|T^+_n(i)|^2 \right]\leq \frac{2n^2}{i^2}\, \cdot 
    \]
\end{lemma}
\begin{proof}[Proof of Theorem~\ref{lem:size-Ti}]
    For $ \Lambda\in A,B$, let $|T^+_{n, \Lambda}(i)|$ be the number of nodes of type $ \Lambda$ in $T_n^+(i)$.  Letting also $\overline \Lambda$ be the complement
of $ \Lambda$ in $\{A, B\}$, define
    \[
    w_{n, \Lambda}^i=(1-q)\frac{|T^+_{n, \Lambda}(i)|}{n}+q\frac{|T^+_{n,\bar{ \Lambda}}(i)|}{n}\, ,
    \]
    the conditional probability, given $\cF_n$, that node $(n+1, \Lambda)$ attaches in $T_n^+(i)$. We have
    \begin{align*}
        \E\left[ |T_{n+1}^+(i)|^2\,|\, \cF_n\right]&= |T_n^+(i)|^2+(4|T_n^+(i)|+4)w_{n,A}^iw_{n,B}^i\\
        &\hspace{1cm}+(2|T_n^+(i)|+1)\left(w_{n,A}^i(1-w_{n,B}^i)+w_{n,B}^i(1-w_{n,A}^i)\right)\\
        &= \left(1+\frac{2}{n}\right)|T_n^+(i)|^2+\frac{|T_n^+(i)|}{n}+2w_{n,A}^iw_{n,B}^i\, ,
    \end{align*}
    where we used that $w_{n,A}^i+w_{n,B}^i=\frac{|T_n^+(i)|}{n}$. Taking logarithm on both sides and using concavity, one has
    \[
    w_{n,A}^iw_{n,B}^i\leq \left(\frac{|T^+_n(i)|}{2n}\right)^2\leq\frac{1}{2}\left(\frac{|T^+_n(i)|}{2n}\right)^2\, .
    \]
    Hence,
    \[
    \E\left[ |T_{n+1}^+(i)|^2\,|\, \cF_n\right]\leq  \frac{(n+1)^2}{n^2}|T_n^+(i)|^2+\frac{|T_n^+(i)|}{n}\, .
    \]
   Taking expectation and using that $\E\left[|T_n^+(i)|\right]=\frac{n}{i}$ by symmetry, we see by induction that the sequence $u_n=\frac{\E\left[|T_n^+(i)|^2\right]}{n^2}$ satisfies
   \[
   u_{n+1}\leq u_n+\frac{1}{i}\leq u_i+\frac{1}{i}\sum_{j=i+1}^{n+1}\frac{1}{j^2}\leq \frac{2}{i^2}\, \cdot 
   \]
\end{proof}

\subsection*{Acknowledgement}
The authors would like to thank two anonymous referees for their valuable remarks and suggestions.
V.V. is supported by Israeli Science Foundation grant number 1695/20 and ERC grant 83473. She also acknowledges support from the foundation \textit{Ferran Sunyer i Balaguer} and \textit{Institut d'Estudis Catalans}, who helped initiate this research through a PhD student mobility scholarship in 2021.

\appendix
\renewcommand\thefigure{\thesection.\arabic{figure}}    
\setcounter{figure}{0}    
\section{Complexity of fair bisection in recursive trees}
The clustering algorithm of Section 3 finds a bisection of the nodes such that none of the two  parts contains repeated labels (such a bisection is called \textit{fair}) and the number of edges in the cut is roughly less than $2q$. A closely related task is finding a fair bisection that minimises the number of edges in the cut (\textit{min fair bisection}). This is an NP-hard problem, as was shown in~\cite{casel_et_al:LIPIcs.FORC.2023.9} for general trees, not necessarily recursive. However, the general problem of minimum-cut bisection is fixed-parameter tractable with respect to treewidth~\cite{jansen2005polynomial}. In particular, for trees there exists a polynomial time algorithm that solves the task, which comes from a bottom up, dynamic programming approach. In view of these results, it is not clear whether the recursiveness\footnote{the property that the node labels are increasing in any path from a root to a leaf} of our trees  make the task easier and bring it to polynomial time. We will show that this is not true and the problem remains NP-hard even for recursive trees. 

In view of our hardness proof, we do not hope to achieve polynomial time by applying
an existing generic graph algorithm. 
If polynomial time is possible, the algorithm would have to rely either on different model-specific properties, or on the detailed structure of random colorings of BCMRT trees that would
allow to look for fair bisections that are not optimal in terms of cut size but
still are inside the margin, or on something different entirely, for instance spectral techniques. These are all non-trivial tasks and interesting in their own right, hence we do not pursue the matter further in this paper.  

\subsection{The reduction}

The $NP$-hard problem of vertex cover in cubic graphs (VCC) can be reduced to our problem's proxy, minimum fair bisection, upon viewing the proof~\cite[Theorem 6]{bonsma2006complexity}. There, the authors reduce VCC to the \textit{paint-shop problem for words}. In this problem, one is given a word $w=w_1\dots , w_k$, such that for every letter $w_j$ in $w$ there exists exactly one index $i\neq j$ such that $w_i=w_j$. The task is to find a 2-coloring of the word's letters such that two appearances of the same symbol have different color (this we call a \textit{valid} coloring) and the number of color changes as we parse the word from one side to the other is minimised. To reduce VCC to min fair bisection, it is enough to make an addition to the reduction in~\cite{bonsma2006complexity}. Let us note that this strategy was also followed in~\cite{casel_et_al:LIPIcs.FORC.2023.9} for the case of general trees, not necessarily recursive, but their proof does not transfer to our case.

In the proof of~\cite[Theorem 6]{bonsma2006complexity}, a cubic graph $G$ of $n$ nodes is mapped to a sequence from $i=1$ until $i=n$ of subwords $W_i$:
\begin{equation}
A_iD_{1i}D_{2i} ...D_{(i-1)i}E_{a}^iE_{b}^iE_{c}^iB_iB_iD_{i(i+1)}D_{i(i+2)} ...D_{in}A_iC_iC_i~,\label{reduction1}
\end{equation}
placed one next to the other to form one word $W$. The $i$-th subword corresponds to the $i$-th node of $G$ according to a specified ordering (see~\cite{bonsma2006complexity} for more details on that). The symbols $A_i, C_i$ appear twice in $W$, only in the $i$-th block. By construction, the symbols $D_{ij}$ are symmetric in $i,j$, that is, $D_{ij}=D_{ji}$ for all $i\neq j$. Hence, each of them appears twice in $W$, once  in the $i$-th and once in the $j$-th subword. Again by construction, for every $i\in[n]$, $k_i\in\{a,b,c\}$, there exists unique  $j\neq i$, $k_j\in \{a,b,c\}$ such that $E_{k_i}^i=E_{k_j}^j$. In~\cite{bonsma2006complexity}, the sequence $E_{a}^iE_{b}^iE_{c}^i$ is associated to the three edges connected with node $i$ in some specified ordering.

Given $W$, we construct a recursive tree $T^W$, where each node label is repeated twice, in the following way (see also Figure~\ref{NP-hard}). For each symbol appearance in $W$ create a new node and label it with the symbol it's associated with. Note that two different nodes corresponding to the same symbol have the same label. Create also $2n$ nodes with labels $(W_i)_{i\leq n}$ (each $W_i$ associated to a pair of nodes). For each $i\in[n]$, pick an arbitrary node with label $W_i$, mark it, and connect it with edge to both nodes with label $B_i$. Connect all marked nodes to form a path, following the order that appears in the associated subword. Connect the two nodes with label $W_1$ and connect the rest unmarked nodes with label $W_j$, $j\in\{2,\dots , n\}$, to the unmarked node $W_1$. For each subword $W_i$, when $1<i<n$ (resp. when $i=1$ or $i=n$), consider the subword beginning at $C_{i-1}A_i$ and ending at $E_{c}^i$ (resp. the subword beginning at $A_1$ and ending at $E_c^1$ and the subword beginning at $C_{n-1}$ and ending at $E_c^n$). In that order, connect with edges the associated nodes so that they form a path. Connect the node with label $E_c^i$ to an arbitrary node labelled $B_i$, and mark that node. Similarly, for $i\not\in \{1,n\}$ (resp. in $W_1$ and $W_n$), consider the subword in $W_i$ that begins with $D_{i(i+1)}$ and ends with the first appearance of $C_i$ (resp. that begins with $D_{12}$ and ends with the first appearance of $C_1$, and that begins with the second appearance of $C_n$ and ends with $A_n$), create a  path graph for each of them, and connect the node with label $D_{i(i+1)}$ to the unmarked node with label $B_i$. 

We now describe a partial ordering of the node labels such that the tree is recursive. Let $W_1$ be the smallest label and $W_i<W_j$ for any $j>i$. For any $i$, let $B_i$ be larger than any $W_j$, and also let any $E^i$, with any subscript, be larger than any $B_j$. Also, let $D_{ij}>D_{ik}$ for any $i<j<k$ and $D_{ij}<D_{ik}$ for any $j<k<i$. Consider all $A_i$ and $C_i$ to be larger than all the aforementioned nodes and $A_i<C_j$ for all $i,j$. To construct an ordering of $\{E_a^i,E_b^i,E_c^i\}_{i\in[n]}$, we need to refer to the specifics of the construction in~\cite{bonsma2006complexity}. There, the authors create a total ordering of the edges of $G$, which, when restricted to edges touching a single node $i$, gives the ordering $E_{a}^i<E_b^i<E_c^i$. For our purpose, it is enough to reverse this ordering. The partial ordering that we  described makes $T^W$ recursive and can be extended to a total ordering.

Now assume a valid 2-coloring with $k$ changes in $W$. Then there exists a valid  coloring with $k+2-n$ changes in $T^W$, by coloring the sequence of nodes $W_1\dots W_n$ that forms a path with the same arbitrary color, the rest $W_i$ with the other color, and the rest of the nodes as they are colored in $W$. Conversely, any valid coloring of $T^W$ with at most $k$ color changes corresponds to a coloring of $W$ with at most $k+n-2$ changes.

\begin{figure}
\vspace*{.2cm}
\includegraphics[scale=.6]{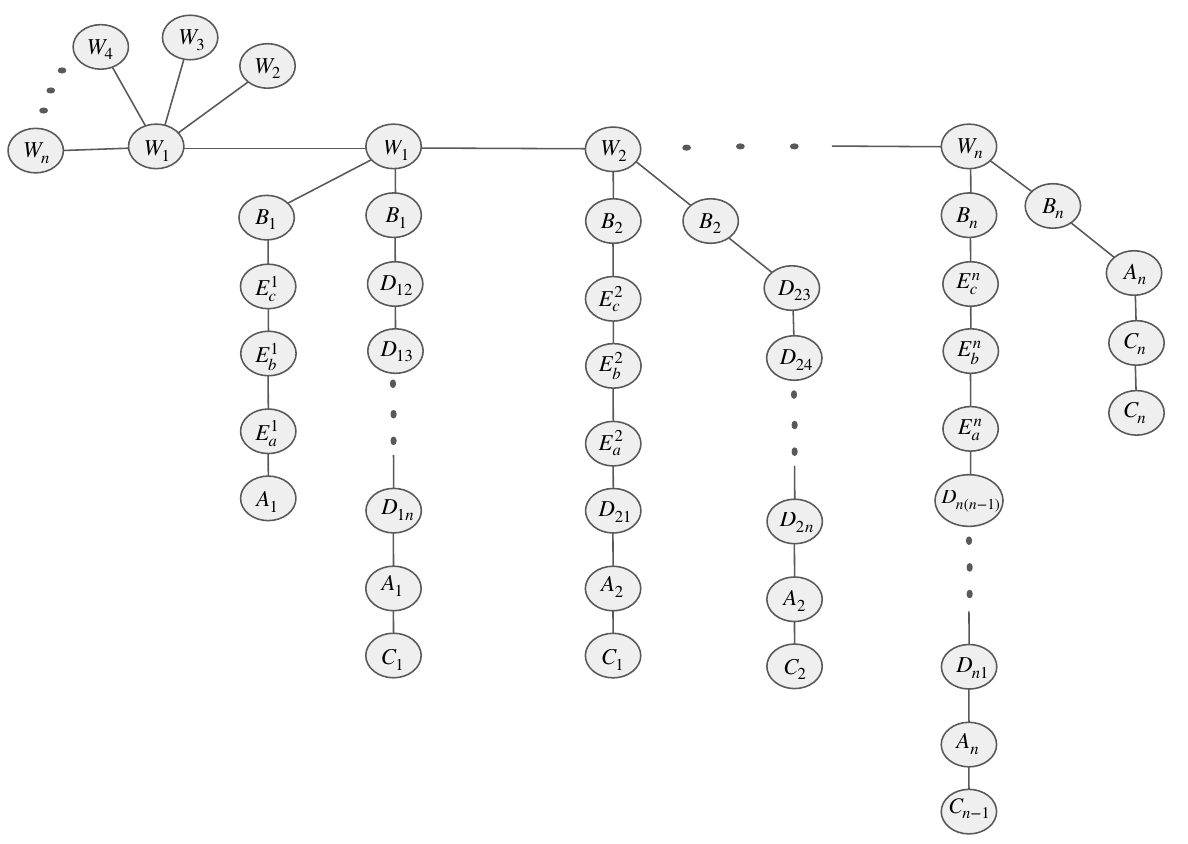}
    \label{NP-hard}
    \caption{The transformation of a word corresponding to a cubic planar graph, of the form~\ref{reduction1}, to a balanced recursive tree.}
\end{figure}

\end{document}